\documentclass[a4paper,reqno,11pt,oneside]{amsart}
\usepackage{graphicx} 
\usepackage{amsmath,geometry}
\usepackage{graphicx}
\usepackage{mathrsfs}
\usepackage{amssymb,amsmath, amsfonts, amsthm}
\usepackage{latexsym}
\usepackage{color} 
\usepackage[dvips]{epsfig}
\usepackage[colorlinks]{hyperref}
\usepackage{bm}

\numberwithin{equation}{section}

\renewcommand{\(}{\left(}
\renewcommand{\)}{\right)}
\newcommand{\beq}{\begin{equation}}
	\newcommand{\eeq}{\end{equation}}
\newcommand{\ba}{\begin{aligned}}
	\newcommand{\ea}{\end{aligned}}

\renewcommand{\d }{\delta }

\renewcommand{\l }{\lambda}

\renewcommand{\O}{\Omega}

\newcommand{\D}{{\mathfrak{D}}}

\newcommand{\Ui}{\U_i}
\newcommand{\Uh}{\U_h}
\newcommand{\Wd}{{W}_{\bm\delta,\bm\xi}}

\newcommand{\E}{\mathcal{E}}
\renewcommand{\H}{{H^1_0(\Om)}}

\definecolor{darkblue}{rgb}{0.05, .05, .65}
\definecolor{darkgreen}{rgb}{0.1, .65, .1}
\definecolor{darkred}{rgb}{0.8,0,0}

\theoremstyle{plain}
\newtheorem{thm}{Theorem}[section]
\newtheorem{lem}[thm]{Lemma}
\newtheorem{prop}[thm]{Proposition}

\numberwithin{equation}{section}
\theoremstyle{remark}
\newtheorem{oss}{Remark}[section]

\DeclareMathOperator{\dist}{dist}
\DeclareMathOperator{\supp}{supp}
\DeclareMathOperator{\vspan}{span}

\renewcommand{\(}{\left(}
\renewcommand{\)}{\right)}

\renewcommand{\O}{\mathcal O}

\newcommand{\eps}{\varepsilon}
\newcommand{\N}{\mathcal N_{\bm\de,\bm\xi}}
\newcommand{\Om}{\Omega}

\renewcommand{\r}{\rangle}
\renewcommand{\l}{\langle}
\renewcommand{\leq}{\leqslant}
\newcommand{\R}{{\mathbb{R}^N}}
\newcommand{\RR}{{\mathbb{R}}}

\renewcommand{\L}{\mathcal{L}}
\newcommand{\U}{\mathcal U}

\newcommand{\e}{\varepsilon}
\newcommand{\de}{\delta}
\newcommand{\la}{\lambda}
\newcommand{\al}{\alpha}
\renewcommand{\i}{ i^\star}

\begin{document}

\title[Nodal cluster solutions for the Brezis-Nirenberg problem]{Nodal cluster solutions for the Brezis-Nirenberg problem in dimensions $N\geq 7$}

\author[M. Musso]{Monica Musso}
\address{M. Musso - Department of Mathematical Sciences, University of Bath, BA2 7AY, UK.}
\email{mm2683@bath.ac.uk}

\author[S. Rocci]{Serena Rocci}
\address{S. Rocci - Dipartimento di Scienze di Base e Applicate per l’Ingegneria,
Sapienza Università di Roma, Via Antonio Scarpa 10, 00161 Roma (Italy) }
\email{serena.rocci@uniroma1.it}

\author[G. Vaira]{Giusi Vaira}
\address{G. Vaira - Dipartimento di Matematica, Università degli Studi di Bari Aldo Moro, Via Orabona 4, 70125 Bari (Italy) }
\email{giusi.vaira@uniba.it}

 \begin{abstract}
     We show that the classical Brezis-Nirenberg problem
     $$
     \Delta u + |u|^{4 \over N-2} u + \varepsilon u = 0 ,\quad {\mbox {in}} \quad \Omega, \quad u= 0 ,  \quad {\mbox {on}} \quad \partial \Omega$$
     admits  nodal solutions  clustering around a point on the boundary of $\Omega$ as $\varepsilon \to 0$, for smooth bounded domains $\Omega \subset \R $ in dimensions $N\geq 7$.
 \end{abstract}
	\maketitle
	
	\section{Introduction}
	In this paper we find a new family of sign-changing solutions to the classical Brezis-Nirenberg problem 
	\beq\label{BN}\tag{BN}
	-\Delta u =|u|^{4 \over N-2}u+\e u\quad \hbox{in}\,\, \Omega,\quad 
		u=0\quad \hbox{on}\,\, \partial\Omega\eeq
	where $\eps>0$ is a small parameter and  $\Omega$ is a smooth bounded  domain in $\mathbb R^N$, $N\geq 7$.

 \medskip
 
 In their seminal 1983 paper \cite{bn}, Brezis and Nirenberg initiated the study of positive solutions to \eqref{BN} and demonstrated that for dimensions $N \geq 4$, the problem  admits a solution for $\eps \in (0, \lambda_{1,\Omega} )$, where $\lambda_{1,\Omega}$ represents the first eigenvalue of $-\Delta$ with $0$-Dirichlet boundary conditions on $\partial \Omega$.  If dimension $N$ is $3$, they proved the existence of $\la_{*,\Omega} >0$ (whose definition depends on $\Omega$) and of a positive solution to \eqref{BN} if $\lambda \in (0,\lambda_{*,\Omega})$. If $\Omega=B$ the unit ball, then
 $\la_{*,B} = {\lambda_{1,B} \over 4}$; for general domains see \cite{dru}.  Multiplying the equation in \eqref{BN} against the eigen-function associated to $\lambda_{1,\Omega}$ and integrating by parts on $\Omega$ show that no positive solutions exist  for $\eps \geq \lambda_{1,\Omega}$. Additionally, Pohozaev identity \cite{poho}  gives that problem \eqref{BN} has no non-trivial solutions 
  when $\Omega$ is star-shaped and $\eps=0$. Conversely, Bahri and Coron \cite{ba-co} presented an existence result for a positive solution to problem (1.1) for $\Omega$ with a nontrivial topology and $\eps =0$. Subsequently, considerable attention has been devoted to understanding the possibility of multiple positive solutions to \eqref{BN} in the regime $\eps \to 0$ \cite{rey1,ba-li-rey, mu-pi,mu-sa,de-do-mu} and also to understanding the limiting behavior of the positive solutions $u_\eps$ of \eqref{BN} as $\eps \to 0$ \cite{han,rey2}. 

  \medskip
Concerning  the existence of sign-changing solutions to \eqref{BN}, this has been established for all range of $\eps >0$: it has been proven  in \cite{CSS} for $\eps \in (0, \lambda_{1,\Omega})$ and $N\geq 6$, and in \cite{CFP} for $\eps \geq \lambda_{1,\Omega}$ and $N\geq 4$. Devillanova and Solimini \cite{de-so} proved the existence of infinitely many sign-changing solutions to \eqref{BN} for any $\eps >0$ when $N\geq 7$. Dimension $7$ seems to be a threshold case as for $4\leq N \leq 6$ there are no radial sign-changing solutions  for \eqref{BN}, when $\Omega$ is a ball and $\eps \in (0, \lambda_{**})$, for some $\lambda_{**}>0$ \cite{at-br-pe}.

\medskip
This paper wants to give a contribution in the understanding of multiple sign-changing solutions to \eqref{BN} in the regime $\eps \to 0$. It is well known that in this regime a crucial role is played by the {\it bubbles}, namely the positive solutions to \eqref{BN} when $\eps=0$ and $\Omega= \R.$
For any $\delta >0$ and $\xi \in \R$, the bubbles
\begin{equation} \label{bubble}\U_{\de,\xi}(x)=\alpha_N\frac{\delta^{\frac{N-2}{2}}}{\left(\d^2+|x-\xi|^2\right)^{\frac{N-2}{2}}},\quad \alpha_N:=\left[N(N-2)\right]^{\frac{N-2}{4}}
\end{equation}
are all the solutions of the problem 
	\beq\label{plim}
 \Delta u + u^{N+2 \over N-2} =0 \quad {\mbox {in}} \quad \R, \quad u \in H^1 (\R).
	\eeq

\medskip
The asymptotic analysis of low-energy sign-changing solutions to \eqref{BN} as $\eps \to 0$ has been studied in \cite{ben-el-pa} for $N\geq 4$:  assuming their existence, such solutions $u_\eps$ have a simple positive and  negative blow-up behaviour at two distinct points of $\Omega$ as $\eps \to 0$, provided the rates of blow-up for the positive and the negative parts are comparable. Roughly speaking, they can be described as follows
$$
\begin{aligned}
u_\eps (x) \sim \U_{\delta_{1\eps }, \xi_{1\eps} }(x) & -  \U_{\delta_{2\eps }, \xi_{2\eps} } (x) \quad {\mbox {with}} \quad \delta_{i\eps}, \, \to 0, \, \xi_{i\eps} \to \xi_i \in \Omega \quad i=1,2, \\
& \quad {\delta_{1\eps} \over \delta_{2\eps} }= O(1), \quad \xi_1\not= \xi_2 , \quad {\mbox {as}} \quad \eps \to 0.
\end{aligned}
$$
Construction, asymptotic analysis and multiplicity of sign-changing solutions exhibiting this type of simple blow-up as $\eps \to 0$ were obtained in \cite{ca-clapp,mi-pi,ba-mi-pi}.

However, in the case of the unit ball, the low-energy radial sign-changing solutions obtained in \cite{CSS} do not have a simple blow-up if $N\geq 7$. Indeed,  both their positive and negative parts  blow-up in the form of a positive and a negative bubble both centered at the center of the ball as $\eps \to 0$, with non comparable  rates of blow-up \cite{pre2}. Roughly speaking, in this case solutions look like
$$
\begin{aligned}
u_\eps (x) \sim \U_{\delta_{1\eps }, 0} (x) & -  \U_{\delta_{2\eps }, 0 } (x) \quad {\mbox {with}} \quad \delta_{i\eps} \, \to 0, \quad i=1,2,  \quad \, \\
& \quad {\delta_{1\eps} \over \delta_{2\eps} }= o(1),  \quad {\mbox {as}} \quad \eps \to 0.
\end{aligned}
$$
This behaviour is known as {\it tower of bubbles} (see \cite{de-do-mu0}). In \cite{pre2} it is proven that sign-changing tower of bubbles for \eqref{BN} exist as $\eps \to 0$ for dimensions $N\geq 7$ in a general domain. In contrast,  in low dimensions $N = 4, 5, 6,$ sign-changing bubble-towers cannot exist, as shown in \cite{ia-pa}. 

\medskip
In \cite{va} Vaira constructed a different type of sign-changing solutions to \eqref{BN} which blow-up in the form of a concentrated bubble and blow-up occurs at a point of the boundary of $\Omega$. Bubbling at the boundary is not always allowed  \cite{rey2}, and some extra requirement on  the domain $\Omega$ seems to be necessary.
In \cite{va} it is assumed that $\Omega$ is a smooth bounded domain with non-trivial topology  such that the problem 
	\beq\label{BN0}
	 -\Delta u_0 =|u_0|^{4 \over N-2}u_0\quad \hbox{in}\,\, \Omega, \quad u_0=0\quad \hbox{on}\,\, \partial\Omega, \quad u_0>0\quad\hbox{in}\,\, \Omega\eeq
	has a positive solution $u_0$ which is non-degenerate, in the sense that  the following  linear problem
	\begin{equation}\label{non-deg}
			-\Delta v= p|u_0|^{4\over N-2}v \quad \mbox{ in } \quad \Om, \quad 
			v=0 \quad \mbox{ in } \quad  \partial\Om
	\end{equation} admits only the trivial solution $v=0$. Existence of solutions to \eqref{BN0} for domain with non-trivial topology has been obtained by \cite{ba-co}. Besides, for generic $\Omega$ these solutions are non-degenerate \cite{sa-te}.

 Let $\nu$ be the unitary outer normal to $\partial \Omega$. Assuming that the function $\xi \in \partial \Omega \mapsto \nabla u_0(\xi) \cdot \nu (\xi) $ has a non-degenerate critical point $\xi_0$, Vaira proves the existence of a sign-changing solution  to problem \eqref{BN} of the form 
 $$
 \begin{aligned}
     u_\eps (x) &\sim u_0 (x) - \U_{\delta , \xi} (x), \quad {\mbox {with}} \quad \delta \sim\eps^{\frac{2(N-1)}{N^2-6N+4}} \\
     & \xi -\xi_0 \sim \eps^{\frac{N-2}{N^2-6N+4}} \quad {\mbox {as}} \quad \eps \to 0.
 \end{aligned}$$
 Here $\U_{\delta , \xi}$ is again the bubble defined in \eqref{bubble}.

 \medskip
 The main result of this paper is to prove that a sign-changing cluster solution to \eqref{BN} around $\xi_0$ is possible. Clustering configurations
are those where the solutions blow-up as the sum of a finite number of bubbles, of comparable heights,
whose centers converge to the same point. Clustering configurations are known to exist in several problems related to semi-linear elliptic equations with critical non-linearity, but none was known for the Brezis-Nirenberg problem \eqref{BN}.

\medskip
To state our result, let us 
 denote by $PW$ the projection of a function $W$ onto $H^1_0(\Omega)$, i.e. 
	$$\Delta PW=\Delta W\quad \hbox{in}\,\, \Omega, \qquad PW=0\quad \hbox{on}\,\, \partial\Omega.$$

 \medskip
Our main result is the following

\begin{thm}\label{teo1}
Let $\Omega$ be a smooth bounded domain in $\R$ with $N\geq 7$, such that Problem \eqref{BN0} has a  solution $u_0$, which is non-degenerate in the sense that the linear problem \eqref{non-deg} has only the trivial solution. Assume there exists a  critical point $\xi_0 \in \partial \Omega $ of the function for $\xi \in \partial \Omega \to \nabla u_0 (\xi) \cdot \nu (\xi)$, where $\nu$ is the unitary outer normal to $\partial \Omega$, such that the second variation $D_{N-1}^2 \left( \nabla u_0 (\xi) \cdot \nu (\xi)\right)$
is positive definite.

Let $k \in \mathbb{N}$. Then there exist $\bar \eps >0$ and a constant $C>0$ such that, for all $\eps \in (0, \bar \eps)$ there exists a sign-changing solution $u_\eps $ to \eqref{BN} given by
$$
u_\eps (x) = u_0 (x) - \sum_{j=1}^k P \U_{\delta_{j \eps} , \xi_{j \eps}} (x) + \phi_\eps (x) $$
where
$$
\delta_{j\eps}  = \eps^{2(N-1) \over N^2-6N+4}\, d_{j \eps} , \quad \xi_{j \eps} \not= \xi_{i \eps} \quad {\mbox {for}} \quad i \not= j, \quad \xi_{j\eps} = \xi_0 + \eps^{N-2 \over N^2-6N+4} \hat \xi_{j\eps} \in \Omega
$$
with
$$
|d_{j\eps}| \leq C, \quad |\hat \xi_{j\eps} | \leq C \quad \forall j=1, \ldots , k,
$$
and
    $$
\| \phi_\eps \|_{H_0^1 (\Omega )}  \leq C \eps^{\frac{N^3-8N+8}{2N(N^2-6N+4)}  + \sigma}
$$
for any $\sigma >0$ arbitrarily small.
\end{thm}
 
 \medskip
The solutions described in the theorem are rather delicate to capture, and precise expansions of the parameters $\delta_{j\epsilon}$ and the points $\xi_{j\epsilon}$ at two consecutive scales are required in the construction. This is described in details in Section \ref{ans}.

	\medskip
The method we use to prove Theorem \ref{teo1} also applies to the construction
of sign-changing solutions exhibiting a cluster configuration near the boundary of $\Omega$ for the almost critical problem
$$
	-\Delta u =|u|^{{4 \over N-2}-\eps}u\quad \hbox{in}\,\, \Omega,\quad 
		u=0\quad \hbox{on}\,\, \partial\Omega
  $$
	where $\eps>0$ is a small parameter and  $\Omega$ is a smooth bounded  domain in $\mathbb R^N$, $N\geq 7$.
This observation is already present in \cite{va} and we will not elaborate further on this point.

 \medskip
 Clustering configurations are known in the literature for perturbation of the Yamabe problem to find metrics on Riemannian manifolds with constant scalar curvature. These have been found in high dimensions $N\geq 7$ in \cite{pi-va1}, in dimensions $4$ and $5$ in \cite{thizy-ve}, see also \cite{ro-ve}. We dont't know if clustering sign-changing solutions   exist for the Brezis-Nirenberg problem \eqref{BN} in low dimensions $4,5,6$, but if it does the form of the solution should though be different from the one obtained in Theorem \ref{teo1}. 

 \medskip
 Finally, we mention that several interesting results have been obtained on the
existence of sign changing solutions to the Brezis-Nirenberg problem in regimes different from the one treated in this paper, namely when $\eps$ converges to some fixed $\eps_* >0$. Results in this direction are contained for instance  in \cite{iava,pi-va2}.

 \medskip
 \noindent

{\bf Acknowledgments.} The authors would like to express their gratitude to Angela Pistoia for many interesting discussions around this topic.
M. Musso has been supported by EPSRC research Grant EP/T008458/1 while G. Vaira has been supported by Gnampa project "Proprietà qualitative delle soluzioni di equazioni ellittiche".
	
	\section{The setting of the problem } \label{ans}

  We consider the Hilbert space $H^1_0(\Om)$ equipped with the usual inner product \begin{equation*}
		\l u, v \r = \int_\Om \nabla u \cdot\nabla v
	\end{equation*} which induces the norm
	\begin{equation*}
		\|u\|_{H^1_0(\Om)}= \(\int_\Om |\nabla u|^2\)^{\frac 12}.
	\end{equation*} For $r\in[1,\infty)$ and $u\in L^r(\Om)$ we set $|u|_{r,\Om} =
	\(\int_\Om |u|^r\)^{\frac 1r}.$
	
	Let $\i: L^{\frac{2N}{N+2}}(\Om) \to H^1_0(\Om)$ be the adjoint operator  of the immersion $ i: H^1_0(\Om) \hookrightarrow L^{\frac{2N}{N-2}}(\Om)$. By definition $u=\i(f)$ if and only if \[\l u,\varphi\r = \int_\Om f\varphi \mbox{ for all } \varphi\in H^1_0(\Om)\]
	or equivalently $u$ weakly  solves \[-\Delta u = f \mbox{ in } \Om , \quad u=0 \mbox{ on } \partial\Om .\] 
	The operator $\i: L^{\frac{2N}{N+2}}(\Om) \to H^1_0(\Om)$ is continuous as \beq \label{ibound} \|\i(f)\|_{H^1_0(\Om)}\leq S^{-1}|f|_{\frac{2N}{N+2};\Om}\eeq where $S$ is the best constant for the Sobolev embedding. 
	 
	In terms of the operator $\i$, problem \eqref{BN} can be formulated as \begin{equation}
		u=\i(|u|^{p-1}u+\eps u) \label{prob} .
	\end{equation}
	We look for cluster solutions of the problem \eqref{BN} which change sign. They  have the form 
	\begin{equation}\label{dd}
 u_\e(x)= W_{\bm\de,\bm\xi}(x)+\phi_{\bm\de,\bm\xi}(x) \quad \mbox{ where } \quad W_{\bm\de,\bm\xi}(x)=u_0(x)-{\sum_{i=1}^k P\Ui(x)}.
 \end{equation}
 Here $k$ is a fixed given integer, $u_0$ is the positive non-degenerate solution to \eqref{BN0}, 
 \[P\Ui=P\U_{\de_i,\xi_i}= \i(\Ui^p), \quad \bm\de = \(\de_1,\cdots,\de_{k}\) \in \RR^k \quad {\mbox {and}} \quad \bm\xi =\(\xi_1,\cdots,\xi_{k}\)\in \Om^k.
 \]
 In our construction the scaling parameters $\de_i$ will be positive and small, while the points $\xi_i$  will collapse into each other, as $\e \to 0$.

	In \cite{va} Vaira constructs a solution to problem \eqref{BN} of the form \eqref{dd}, with $k=1$ under the assumption that there exists   a non-degenerate critical point $\xi_0$ of the function $\xi \in \partial \Omega \mapsto \partial_\nu u_0(\xi)$, where $\nu$ is the inner unit normal on the boundary. Such solution blows-up, as $\e \to 0$, at $\xi_0$, in the sense that 
 the scaling parameter $\delta$ and the point $\xi$ in \eqref{dd}  can be described at main order as
	\begin{equation} \label{d0}
		\begin{aligned}
			&\delta \sim \eps^\alpha d_0 &\mbox{ with } \alpha =\frac{2(N-1)}{N^2-6N+4} \\
			&\xi\sim \xi_0+t_0 \eps^\beta \nu(\xi_0) &\mbox{ with } \beta =\frac{N-2}{N^2-6N+4} ,
		\end{aligned} \quad {\mbox { as }} \eps \to 0.
	\end{equation}
One has $\delta \to 0$ and $\xi \to \xi_0$ as $\eps \to 0.$	 The point $(d_0,t_0,\xi_0) \in \mathbb{R}^+ \times \mathbb{R} \times \partial \Omega $ is a critical point of the function
	\beq \label{dd0} \Psi(d,t,\xi)=-\mathtt C \partial_\nu u_0(\xi) d^\frac{N-2}{2}t+\frac{\alpha_N}{2^{N-1}}\mathtt C \frac{d^{N-2}}{t^{N-2}}-\mathtt B d^2 \eeq
	where $\mathtt C$ and 
	 $\mathtt B$ are the explicit positive constants
	\beq \mathtt C= \int_\R \U_{1,0}^p \quad and  \quad \mathtt B= 
	\frac 12\int_\R \U_{1,0}^2 . \eeq
 A direct computation gives that $(d_0,t_0,\xi_0)$ satisfies the system
	\begin{equation}\label{sys}
		\left\{\begin{aligned}&-2\mathtt B d_0+\frac{\alpha_N}{2^{N-1}}\frac{d_0^{N-3}}{t_0^{N-2}}\mathtt C-\mathtt C\frac{N-2}{2}d_0^{\frac{N-4}{2}}t_0\partial_\nu u_0(\xi_0)=0\\
			&-\frac{\alpha_N(N-2)}{2^{N-1}}\frac{d_0^{N-2}}{t_0^{N-1}}\mathtt C-\mathtt C d_0^{\frac{N-2}{2}}\partial_\nu u_0(\xi_0)=0
			\\& \nabla_{\xi_0} \partial_\nu u_0(\xi_0) = 0. 
		\end{aligned}\right.\end{equation}
  Since by Hopf's Lemma $\partial_\nu u_0(\xi_0)<0$, 
the function $\Psi$ has a critical point in the considered region. 

 The result in \cite{va} indicates that a solution with the form \eqref{dd} and $k >1$ would possibly exhibit a cluster behaviour around the point $\xi_0$. This suggests the form for the scaling parameters $\delta_i$ and the points $\xi_i$ in \eqref{dd}. Let us be more precise.
 
 Locally around $\xi_0$, the boundary 
	$\partial\Omega$ can be described as the set of points  
 $x= (x',\vartheta(x'))$, for a certain smooth function $\vartheta: \RR^{N-1} \to \RR$.  Without loss of generality, we assume  $\partial_{x_j}\vartheta (\xi_0)=0$ for all $j=1,\cdots, N-1$ and we write $\xi_0 = (\xi_0', \vartheta (\xi_0'))$.
For the construction of a cluster solution, we assume that $\bm\xi = (\xi_1,\cdots,\xi_k) \in \Om^k$ and $\bm\de=(\de_1,\cdots,\de_k)\in\RR^k$ in \eqref{dd} have the form
	\beq \label{deltai}\begin{aligned}  &\d_i=\e^{\alpha}d_0+\e^{\hat\alpha}d_i &\mbox{ with } \hat\alpha=\frac{3N^2-6N+4}{N(N^2-6N+4)}; \\
		&\hat\xi_i= (\xi_0'+\eps^{\hat\beta}\tau_i, \vartheta(\xi_0'+\eps^{\hat\beta}\tau_i)) &\mbox{ with } \hat\beta= \frac{(N-2)^2}{N(N^2-6N+4)};\\
		&\xi_i=\hat\xi_i +(\e^\beta t_0+\e^{\tilde \beta}t_i) \nu(\hat\xi_i) &\mbox{ with }
		\tilde\beta:=\frac{2N^2-6N+4}{N(N^2-6N+4)} .\end{aligned} \eeq
	where $d_0,t_0, \xi_0$, $\alpha$ and $\beta$ are defined in \eqref{d0}, $\hat{\xi}_i\in\partial\Omega$, $\tau_i\in\RR^{N-1}$, $\tau_i , d_i \in \mathbb{R}$ are parameters to be found. For the moment,  we make the following assumptions on these parameters: we assume there exist $a>0$ and $\rho >0$ such that
 $\tau_i$, $\tau_i$ and $d_i$
 \beq\label{par}
 |d_i| , \, |t_i | < a, \forall i=1, \ldots , k, \quad |\tau_i - \tau_h | > \rho \quad \forall i \not= h.
 \eeq

 {	\begin{oss}\label{distxi}
		Without loss of generality, we can choose a coordinate system such that $\nabla_{N-1} \vartheta (\xi_0') = 0$. Now a computation shows that  { $$ \ba |\hat\xi_i-\hat\xi_h|^2 &= \eps^{2\hat\beta}\sum_{\ell=1}^{N-1} (\tau_i-\tau_h)^2_\ell + (\vartheta(\xi_0'+\eps^{\hat\beta}\tau_i)-\vartheta(\xi_0'+\eps^{\hat\beta}\tau_h))^2 \\&= \eps^{2\hat\beta} |\tau_i-\tau_h|^2 + \eps^{2\hat\beta} \underbrace{\nabla_{N-1} \vartheta (\xi_0')}_{=0}\cdot (\tau_i-\tau_h) + o(\eps^{2\hat\beta}). \ea $$ and then
			\beq \label{diff xi}\begin{aligned}|\xi_i-\xi_h|^2&=|\hat\xi_i+(\e^\beta t_0+\e^{\tilde\beta} t_i)\nu(\hat\xi_i)-\hat\xi_h-(\e^\beta t_0+\e^{\tilde\beta} t_h)\nu(\hat\xi_h)|^2\\
				&= |\hat\xi_i-\hat\xi_h+o(\eps^{\hat\beta})|^2\\
				&=\e^{2\hat\beta}|\tau_i-\tau_h|^2 + o\left(\e^{2\hat\beta}\right).\end{aligned}\eeq}
	\end{oss}}

	It is important to observe that $\alpha<\hat\alpha$ and $\hat\beta<\beta<\tilde\beta$. Let us call
	\beq \theta= 1+2\alpha = (\alpha-\beta)(N-2)= \frac{N(N-2)}{N^2-6N+4} \label{theta},\eeq and
	\beq \hat\theta = 1+2\hat\alpha = (\alpha-\hat\beta)(N-2) = \frac{N^3-8N+8}{N(N^2-6N+4)}  \label{htheta}.\eeq

	Now we are able to state our main result.
	
	\begin{thm}\label{main thm}
		Assume there exists  is a $C^1$-stable critical point $\xi_0$ of the function $\xi \in \partial \Omega \mapsto \partial_\nu u_0(\xi)$  such that $D^2_{N-1}\partial_\nu u_0(\xi_0 )$ is positive definite, and let $k \in \mathbb{N}$. Then there exists $\eps_0>0$ such that for any $\eps \in(0,\eps_0)$ the problem \eqref{BN} has a sign-changing cluster solution which blows-up at $\xi_0$. More precisely, there exist constants $C$, $a$, $\rho$, a function 
$\phi_\eps \in H^1_0 (\Omega )$, points $\bm\xi_\eps = (\xi_{1\eps} ,\cdots,\xi_{k \eps}) \in \Om^k$ with $\xi_{i\eps} \not= \xi_{h \eps}$ if $i\not= h$, and parameters $\bm\de_\eps =(\de_{1 \eps},\cdots,\de_{k \eps})\in\RR^k$ satisfying \eqref{deltai}-\eqref{par} such that $u_\eps$ defined in \eqref{dd} is a solution to \eqref{BN}. Moreover 
$$
\| \phi_\eps \|_{H_0^1 (\Omega )}  \leq C \eps^{{\hat \theta \over 2} + \sigma}
$$
for some $\sigma >0$ arbitrarily small.
  
	\end{thm}

Theorem \ref{teo1} is a direct consequence of Theorem \ref{main thm}.

The rest of the paper is devoted to prove Theorem \ref{main thm}. The proof is done via a reduction method.

In Section \ref{reduction} we prove that, for given $\bm\xi = (\xi_1,\cdots,\xi_k) \in \Om^k$ and $\bm\de=(\de_1,\cdots,\de_k)\in\RR^k$ satisfying \eqref{deltai}-\eqref{par}, there exists  $\phi_\eps$   solution of a projected problem. We then show that a true solution to our problem can be achieved by finding a specific set of parameters $\bm\xi = (\xi_1,\cdots,\xi_k) \in \Om^k$ and $\bm\de=(\de_1,\cdots,\de_k)\in\RR^k$. This is the reduced finite dimensional problem that we treat in the subsequent sections.

	\section{Reduction to a finite dimensional problem}\label{reduction}

The purpose of this section is to find 
the term $\phi_\e$ in \eqref{dd} for given $\bm\xi = (\xi_1,\cdots,\xi_k) \in \Om^k$ and $\bm\de=(\de_1,\cdots,\de_k)\in\RR^k$ satisfying \eqref{deltai}-\eqref{par}. The term $\phi_\eps$ will be small in $\eps$ and  satisfy a set of orthogonal conditions.   To introduce these orthogonality conditions,  consider the linear problem
	\[-\Delta \psi = p \U_{1,0}^{p-1} \psi \mbox{ in } \R,\]
	whose set of solution is spanned by the functions 
	\[\psi^0(x) = \frac{N-2}{2}\U_{1,0}(x) +  \nabla \U_{1,0}(x) \cdot x= \alpha_N \frac{N-2}{2} \frac{1-|x|^2}{(1+|x|^2)^{N/2}}  \] and
	\[\psi^j(x) = \frac{\partial \U_{1,0}}{\partial x_j} (x) = -\alpha_N (N-2) \frac{x_j}{(1+|x|^2)^{N/2}} . \]
	Set \[\psi_{\de,\xi}^j (x) = \frac{1}{\de^{\frac{N-2}{2}}}\psi^j\(\frac{x-\xi}{\de}\) \mbox{ for all }   j=0,\cdots, N\] 
	and  define  \[P\psi_i^j = \i \left(p\Ui^{p-1}\psi_i^j\right) \mbox{ for all } i=1,\cdots,k \mbox{ and } j=0,\cdots,N\] where $\Ui=\U_{\de_i,\xi_i}$ and $\psi_i^j=\psi_{\de_i,\xi_i}^j$.
	For any $\bm\de\in\RR^k$ and $\bm\xi\in\Omega^k$, consider the space
	\[K_{\bm\de,\bm\xi} = \vspan \{P\psi_i^j \mbox{ for all }  i=1,\cdots, k \mbox{ and } j=0,\cdots, N\} \] and its orthogonal 
	\[K^\perp_{\bm\de,\bm\xi} = \{ \phi \in H^1_0(\Om) : \l \phi, P\psi_i^j \r = 0 \mbox{ for all }  i=1,\cdots, k \mbox{ and }  j=0,\cdots, N  \} .\]
	The reminder term introduced in \eqref{dd} will satisfy $\phi_\eps\in K_{\bm\de,\bm\xi}^\perp$.
	
	Now let us recall some useful results.

\begin{lem}
		Let $N>6$ and $q \in \left(\frac{p+1}{2},p+1\right]$. 
		Let $\de\in(0,\infty)$, $\xi\in \Om$ and $\U_{\delta , \xi}$ as in \eqref{bubble}.   Then there exist there exists a constant $C=C(a)>0$ and $\eps_0>0$ such that for any $\eps\in (0,\eps_0)$ 
		it holds 	\beq |\U_{\de,\xi}|_{2^\star,\Om} \leq C, \label{est U2star}  \eeq
		\beq |\U_{\de,\xi}|_{\frac{2N}{N+2},\Om} \leq C \de^2, \label{est Uconj}\eeq
		\beq |\psi^j_{\de,\xi}|_{\frac{2N}{N+2},\Om} \leq C\de^2 . \label{psi conj}\eeq
	\end{lem}
	
	Let us denote by $G(x, y)$ the Green's function of the Laplace operator with Dirichlet boundary condition and let $H(x, y)$ be its regular part, namely 
	$$H(x, y)=\gamma_N\left(\frac{1}{|x-y|^{N-2}}-G(x, y)\right), \quad\forall\,\, x, y\in\Omega,\,\, \hbox{with}\,\, \gamma_N=\frac{1}{(N-2)| \mathbb S^{N-1}|}.$$ 
 {We remark that
	in \cite{rey1} it was shown that for every $\xi\in\Omega$
	$$H(\xi, \xi)=\frac{1}{2^{N-2}{\rm dist}(\xi, \partial\Omega)^{N-2}}+ {o\left(\frac{1}{{\rm dist}(\xi, \partial\Omega)^{N-2}}\right)}.$$
 By using the previous observations we get that $$H(\xi_h, \xi_i)=\frac{1}{2^{N-2}}\frac{1}{\eta_i^{N-2}} + o \(\frac{1}{\eta_i^{N-2}}\)$$ where \beq \eta_i = \dist(\xi_i,\partial \Om) = \e^\beta t_0+\e^{\tilde\beta}t_i .\label{etai}\eeq
	}

	\begin{lem} \label{yan0}
		Let $\de\in(0,\infty)$, $\xi\in \Om$ and $\varphi_{\de,\xi}=\U_{\de,\xi}-P\U_{\de,\xi}$. We have
		\[ \varphi_{\de,\xi}(x) =\alpha_N\de^\frac{N-2}{2}H(\xi,x)+ \O\(\frac{\de^\frac{N+2}{2}}{\dist(\xi,\partial\Om)^N}\) 
  \]
  and \[ 0 \leq \varphi_{\de,\xi}\leq \U_{\de,\xi}.
  \]
		Moreover
		\[ \psi^0_{\de,\xi} - P\psi^0_{\de,\xi} = \alpha_N\frac{N-2}{2}\de^\frac{N-2}{2}H(\xi,x)+ \O\(\frac{\de^\frac{N+2}{2}}{\dist(\xi,\partial\Om)^N}\) \] and for all $j=1,\cdots,N$
		\[ \psi^j_{\de,\xi} - P\psi^j_{\de,\xi} = \alpha_N\de^\frac{N}{2} \frac{\partial H}{\partial \xi_j}(\xi,x)+ \O\(\frac{\de^\frac{N+4}{2}}{\dist(\xi,\partial\Om)^{N}}\). \]
	\end{lem}
	\begin{proof}
		See Proposition 1 in \cite{rey1}. 
	\end{proof}
	
	In particular Lemma \ref{yan0} implies that $P\U_{\de,\xi} \geq 0$ and that there exists a positive constant $C$ such that
	\beq \label{varphi infty}|\varphi_{\de,\xi}|_{\infty,\Om} \leq C \frac{\de^\frac{N-2}{2}}{\dist(\xi,\partial\Om)^{N-2}} . \eeq
	
	In \cite{rey1} it is also shown that \beq \label{varphi 2star}|\varphi_{\delta,\xi}|_{{2^\star},\Omega}\lesssim \frac{\delta^{\frac{N-2}{2}}}{{\rm dist}(\xi, \partial\Omega)^{\frac{N-2}{2}}}.\eeq
	More generally we can estimate the $L^q(\Omega)$ norm of $\varphi_{\de,\xi}$ and the $L^\frac{2N}{N-2}(\Omega)$ of $P\psi_{\de,\xi}^j-\psi_{\de,\xi}^j$.
	\begin{lem} \label{psi diff}
		Let $N>6$, $q \in \left(\frac{p+1}{2},p+1\right]$ and $\bm\de,\bm\xi$ satisfy \eqref{deltai} and \eqref{par} {as in Section \ref{ans}}.  Then there exists a constant $C= C(a)>0$  and $\eps_0>0$ such that for any $\eps\in (0,\eps_0)$ for all $i=1,\cdots,k$ and $j=0,\cdots,N$ it holds \[|\varphi_{\de_i,\xi_i}|_{q,\Om}\leq C \eps^{\frac{\theta}{2}}\] and  
  \[|P\psi_i^j-\psi_i^j|_{\frac{2N}{N-2},\Om}\leq C \eps^\sigma\] for some $\sigma>0$ arbitrarily small.
	\end{lem}
	\begin{proof}
		The proof easily follows from Lemma 2.2 in \cite{va}.
	\end{proof}
 
	Now let us define the operators
	\[\Pi (u) = \sum_{i=1}^k \sum_{j=0}^N\l u,P\psi_i^j \r P\psi_i^j\] and
	\[\Pi^\perp (u) = u-\Pi(u) .\]
	Moreover the problem \eqref{prob} is equivalent to solve the system
	\begin{equation}\label{eq pi}
		\Pi \{ u_\eps -\i\(f(u_\eps)+\eps u_\eps\) \} =0
	\end{equation} and
	\begin{equation}\label{eq piperp}
		\Pi^\perp \{ u_\eps -\i\(f(u_\eps)+\eps u_\eps\) \}=0.
	\end{equation}

 Let us also define the linear operator
	\beq \label{L def}\L_{\bm\de,\bm\xi}\phi = \Pi^\perp\left\{\phi-\i\left[ f'\(u_0-\sum_{h=1}^k \Uh\)\phi \right]\right\} , \eeq
	the error term	
	\begin{equation} \label{err}
		\E_{\bm\de,\bm\xi} = \Pi^\perp \left\{\Wd -\i \left[ f\(\Wd\)+\eps\Wd \right]\right\} 
	\end{equation}
	and the nonlinear terms
	\begin{align}
		&\mathcal N^1_{\bm\de,\bm\xi}\phi =\Pi^\perp \{\i \left[f(\Wd+\phi) - f(\Wd)-f'(\Wd)\phi +\eps\phi \ \right]\},  \label{N1}\\
		&\mathcal N^2_{\bm\de,\bm\xi}\phi =\Pi^\perp \left\{\i \left[\(f'\(u_0-\sum_{h=1}^k P\Uh\) - f'\(u_0-\sum_{h=1}^k \Uh\)\)\right]\phi\right\},\label{N2} 
	\end{align}
	where \normalcolor{$f(u)=|u|^{p-1}u$}.

	\begin{oss}
		We want to stress that $\Pi^\perp_{\bm\de,\bm\xi}$ is a continuous map, i.e. there exists a constant $C>0$ such that for any $\eps>0$, $\bm\de\in\RR^k$, $\bm\xi\in\Om^k$ it holds
		\[ \|\Pi^\perp_{\bm\de,\bm\xi} u\|_{\H}\leq C \|u\|_\H \quad \mbox{ for all } \quad u\in\H.\]
	\end{oss}
	
	Two important inequalities we will use throughout all our work are the following.
	\begin{lem}\label{yan} For all $a>0$ and $b\in \RR$, we have
		\[||a+b|^r-a^r| \leq \begin{cases}
			C(r)\min\{a^{r-1}b, b^r\} &\mbox{ if } r<1,\\
			C(r)\(a^{r-1}b+b^r\) &\mbox{ if } r\geq 1
		\end{cases} \]
		and 
		\[||a+b|^r(a+b)-a^{r+1}-(1+r)a^rb| \leq \begin{cases}
			C(r)\min\{|b|^{r+1}, a^{r-1}b^2\} &\mbox{ if } 0\leq r<1,\\
			C(r)\(|b|^{r+1},a^{r-1}b^2\) &\mbox{ if } r\geq 1.
		\end{cases} \]
	\end{lem} \begin{proof}
		See Lemma 2.2 in \cite{yan}.
	\end{proof}
	
\begin{color}{blue}	
	
	\end{color}
	
	Now we want to solve \eqref{eq piperp}. For this aim we start proving the invertibility of $\L_{\bm\de,\bm\xi}$ defined in \eqref{L def}.
	\begin{prop} \label{prop L inv}
		Let $N>6$ and $\bm\de,\bm\xi$ satisfy \eqref{deltai} and \eqref{par} {as in Section \ref{ans}}. Then 
		there exists $C= C(a)>0$ and $\eps_0>0$ such that for any $\eps \in (0,\eps_0)$ 
		it holds \begin{equation}\label{L in}\|\L_{\bm\de,\bm\xi}(\phi)\|_\H\geq C\|\phi\|_{\H} \mbox{ for any } \phi \in K_{\bm\de,\bm\xi}^\perp .\end{equation} 
		Furthermore, the operator $\L_{\bm\de,\bm\xi}$ is invertible and its inverse is continuous.
	\end{prop}
	\begin{proof} We will argue as in the proof of Lemma 1.7 in \cite{mu-pi}. The main difference is that our points are converging to the same point $\xi_0\in \partial\Om$ while in \cite{mu-pi} they are converging to different points in $\Omega$.
		In view of a contradiction, assume that there exist
		\begin{itemize}		\item[$\cdot$] sequences $\eps_n \to 0$ as $n \to \infty$, $|d_i|,|t_i|<a$ 
  for all $i=1,\cdots, k$ as $n \to \infty$, 
			\item[$\cdot$]  $\phi_n\in K^\perp_n := K_{\bm\de_n,\bm\xi_n}^\perp$ ,  with $\|\phi_n\|_\H = 1$; \item[$\cdot$] $\L_n\phi_n :=\L_{\pmb{\de_n},\pmb{\xi_n}}(\phi_n)=h_n$ with $\|h_n\|_\H\to 0 .$ 
		\end{itemize}
		We have that 
		\begin{align}\label{phin eq}
			\phi_{n} = i^*\left[f'\(u_0-\sum \U_{h_n}\)\phi_{n}\right]+h_n+w_n
		\end{align} where $w_{n} \in K_n = K_{\bm\de_n,\bm\xi_n}$, i.e. 
  $$ w_{n} = \sum_{i=1}^k \sum_{j=0}^N c_{i_n}^j P\psi_{i_n}^j. $$
		Observe that  \beq \label{diffN2} |f'\(\sum \Uh\)-\sum f'\(\Uh\)|_{\frac N2,\Om} = \O\(\eps^{2(\alpha-2\beta)}\).\eeq Indeed 
		\begin{align*}
			|f'\(\sum \Uh\)-\sum f'\(\Uh\)|_{\frac N2,\Om}  \leq & C  |f'\(\sum \Uh\)-\sum f'\(\Uh\)|_{\frac N2 \Omega \backslash \cup B_{\eta_i}{\xi_i}} 
			\\&+ C \sum_i \(|f'\(\sum_{h=1}^k \Uh\)-f'(\Ui)|_{\frac N2,B_{\eta_i}{\xi_i}} + \sum_{h\neq i} |f'(\Uh)|_{\frac N2,B_{\eta_i}{\xi_i}}\)  \\
			\leq &  C \sum_i \(|f'\(\sum_{h\neq i} \Uh\)|_{\frac N2,B_{\eta_i}{\xi_i}} + \sum_{h\neq i} |f'(\Uh)|_{\frac N2,B_{\eta_i}{\xi_i}}\) 
			\\&+ C \sum |f'\(\sum \Uh\)-\sum f'\(\Uh\))|_{\frac N2, \Omega \backslash \cup B_{\eta_i}{\xi_i}} \\
			&\leq C  \sum_{{i,h}}  \( \(\sum_{h\neq i} \(\frac{\de_h}{\eta_h^2}\)^\frac{N-2}{2}\)^\frac{4}{N-2} + \sum_{h\neq i} \(\frac{\de_h}{\eta_h^2}\)^2 \) |B_{\eta_i}{\xi_i}|^\frac{2}{N}  
			\\&+ C \(\sum \(\frac{\de_h}{\eta_h^2}\)^\frac{N-2}{2}\)^\frac{4}{N-2} + \sum \(\frac{\de_h}{\eta_h^2}\)^2
		\end{align*} where $\eta_i$ is defined in \ref{etai}.
		Let us also observe that $w_n\in K_n $ is orthogonal to  $\phi_n,h_n\in K_n^\perp$.
		\begin{itemize}
			\item[$\cdot$] {Step 1}. \textit{It holds that \[ \|w_n\|_{H_0^1(\Om)} \to 0 \mbox{ as } n \to \infty.\]} 
			Observe that \[ 0 = \l\phi_n,w_n\r =\sum_{i=1}^k\sum_{j=0}^N c_{i_n}^j \int_\Om \phi_n f'(\U_{i_n})\psi_{i_n}^j .\]
			Using \eqref{phin eq} and  using Lemma \ref{yan},  we have 
			\begin{align*}	\|w_{n}\|_\H^2&=\underbrace{\l\phi_n,w_n\r}_{=0}- \l i^*\left[f'\left(u_0-\sum_{h=1}^k \U_{h_n}\right)\phi_n\right],w_n\r-\underbrace{\l h_n,w_n\r}_{=0} \\&= -\int_\Om \left[f'(u_0-\sum_{h=1}^k \U_{h_n})-f'\(\sum_{h=1}^k \U_{h_n}\)\right]\phi_n w_n  \\&- \int_\Om \left[f'\(\sum_{h=1}^k \U_{h_n}\)-\sum_{h=1}^k f'(\U_{h_n})\right]\phi_n w_n - \sum_{h=1}^k\int_\Om f'(\U_{h_n}) \phi_n w_n \\&=  -\sum_{i=1}^k\sum_{j=0}^N c_{i_n}^j\int_\Om \left[\underbrace{f'\(u_0-\sum_{h=1}^k \U_{h_n}\)-f'\(\sum_{h=1}^k \U_{h_n}\)}_{\leq f'(u_0)}\right]\phi_n P\psi_{i_n}^j \\&- \int_\Om \left[f'\(\sum_{h=1}^k \U_{h_n}\)-\sum_{h=1}^k f'(\U_{h_n})\right]\phi_n w_n \\&- \sum_{h,i=1}^k\sum_{j=0}^N c_{i_n}^j\int_\Om f'(\U_{h_n}) \phi_n (P\psi_{i_n}^j-\psi_{i_n}^j) +\underbrace{\sum_{h,i=1}^k\sum_{j=0}^N c_{i_n}^j\int_\Om f'(\U_{h_n})\phi_n\psi_{i_n}^j}_{=0}\\&\leq {f'(u_0)}\sum_{h,i=1}^k\sum_{j=1}^N |c_{i_n}^j||\phi|_{\frac{2N}{N-2},\Om} \underbrace{|P\psi_{i_n}^j|_{\frac{2N}{N+2},\Om}}_{\normalcolor{}{\O(\de_{i_n})^2}} \\&+|f'\(\sum_{h,i=1}^k \U_{h_n}\)-\sum_{h,i=1}^k f'(\U_{h_n})|_{\frac{N}{2},\Om} |\phi_n|_{\frac{2N}{N-2},\Om}|w_n|_{\frac{2N}{N-2},\Om}\\&+\sum_{h,i=1}^k\sum_{j=0}^N |c_{i_n}^j| |f'(U_{h_n})|_{\frac N2,\Om} |\phi_n|_{\frac{2N}{N-2},\Om}|P\psi_{i_n}^j-\psi_{i_n}^j|_{\frac{2N}{N-2},\Om}
			\end{align*}
			where $|P\psi_{i_n}^j|_{\frac{2N}{N+2},\Om}$,$|f'(\sum \U_{h_n})-\sum f'(\U_{h_n})|_{\frac{N}{2},\Om}$ and $ |P\psi_{i_n}^j-\psi_{i_n}^j|_{\frac{2N}{N-2},\Om}$ go to zero as $\eps_n \to 0$ by \eqref{psi conj}, \eqref{diffN2} and Lemma \ref{psi diff}   and $|f'( \U_{h_n})|_{\frac{N}{2},\Om}$ is bounded as $\U_{1,0}\in L^\frac{N}{2}(\R)$ and
			\[|f'(\Uh)|_{\frac N2,\Omega} \leq \alpha^{p-1}  |U_{1,0}|_{\frac N2, \R}. \]
			Then we have
			\beq \|w_n\|^2_\H \leq o_n(1) \sum_{j=0}^N \sum_{i=1}^k|c_{i_n}^j| \label{wn1}+ o_n(1)\|w_n\|_\H.\eeq
			\normalcolor{}{Moreover, 
				\beq \label{wn0}\|w_n\|^2_\H = \sum c_{i_n}^j c_{r_n}^s \l P\psi_{i_n}^j,P\psi_{r_n}^s\r = \sum c_{i_n}^j c_{r_n}^s[\de_{j,s}\de_{i,r}+o_n(1)] .\eeq
				Putting together \eqref{wn1} and \eqref{wn0},} we have that $c_{i_n}^j$ are bounded and consequently  $\|w_n\|_\H\to 0 $ as $n \to \infty$ . Now as $\|\phi_n\|$ is bounded in $H_0^1(\Om)$, it converges weakly in $\mathcal{D}^{1,2}(\Om) $ 
			to $\phi_\infty$ . We want to show that $\phi_\infty=0$.
			\item[$\cdot$] Step 2.       
			For all $i=1,\cdots,k$ let us define \[\tilde\phi_{i_n}(y)=\de_{i_n}^{\frac{N-2}{2}}\phi_n(\de_{i_n}y+\xi_{i_n}) \mbox{ for all } y\in\Om_{i_n}= \frac{\Om-\xi_{i_n}}{\de_{i_n}} ,\] such that $\|\tilde{\phi}_{i_n}\|_{H^1_0(\Om_{i_n})} = \|\phi_n\|_{\H} = 1$. Then $\tilde{\phi}_{i_n}$ is bounded and there exists  $\tilde{\phi}_{i_\infty}$ such that  $\tilde{\phi}_{i_n} \to \tilde{\phi}_{i_\infty}$ weakly in $\mathcal D^{1,2}(\R)$. We observe that 
			\[\tilde{\phi}_{i_n} = \de_{i_n}^2 \i_{i_n} \left[  f'\( u_0(\de_{i_n}y+\xi_{i_n})-\de_{i_n}^{-\frac{N-2}{2}}\U_{1,0}-\sum_{h\neq i} \U_h(\de_{i_n}y+\xi_{i_n}) \)\tilde{\phi}_{i_n}  \right] + \tilde h_{i_n}(y)+\tilde w_{i_n}(y)  ,\] 
			where $\i_{i_n}$ is the adjoint operator  of the immersion $ \i_{{i_n}}: H^1_0({\Om_{i_n}}) \hookrightarrow L^{\frac{2N}{N-2}}({\Om_{i_n}})$, $\tilde h_{i_n}(y) = h_n(\de_{i_n}y+\xi_{i_n})$ and  $\tilde w_{i_n}(y) = (\de_{i_n}y+\xi_{i_n}).$ Recalling that $|\xi_i-\xi_h|^2 = \O (\eps^{2\beta})$ by \eqref{distxi}, for all $\varphi \in C_c^\infty$ such that $\supp \varphi \subset \Om_{i_n}$, we have \normalcolor{}{
				\begin{align*}
					\l \tilde{\phi}_{i_n}&,\varphi \r_{H_0^1(\Om_{i_n})} =  \de_{i_n}^2 \int_{\Om_{i_n}} f'\(u_0-\sum\Uh\)\tilde{\phi}_{i_n}\varphi + \l \underbrace{h_n^i+w_n^i, \varphi\r}_{=o_n(1)}\\ \leq& \de_{i_n}^2 \int\limits_{\Om_{i_n}}f'\(u_0(\de_{i_n}y+\xi_{i_n})-\de_{i_n}^{-\frac{N-2}{2}}\U_{1,0}-\al_N\sum_{h \neq i} \frac{\de_h^{\frac{N-2}{2}}}{ (\de_h^2+|\de_i y +\xi_{i_n}-\xi_{h_n}|^2)^{\frac{N-2}{2}} }\) \tilde{\phi}_{i_n} \varphi  \\&+o_n(1)
					\\=&\int\limits_{\Om_{i_n}}\left[f'\(\U_{1,0}(y)-\de_{i_n}^{\frac{N-2}{2}}u_0(\de_{i_n}y+\xi_{i_n})+\al_N\sum_{h\neq i} {\frac{\de_{h_n}^{\frac{N-2}{2}}\de_{i_n}^{\frac{N-2}{2}}}{(\de_h^2+|\de_i y+\xi_{i_n}-\xi_{h_n}|^2)^\frac{N-2}{2}})} \) \right. \\ & \big.-f'(\U_{1,0}) \Bigg]\tilde{\phi}_{i_n}\varphi   +p\int\limits_{\Om_{i_n}} \U_{1,0}^{p-1}\tilde{\phi}_{i_n} \varphi+o_n(1)
					\\ \leq&  \int\limits_{\Om_{i_n}} f'\(\de_{i_n}^\frac{N-2}{2}u_0(\de_{i_n}y+\xi_{i_n})-\al_N\sum_{h\neq i} \underbrace{\frac{\de_{h_n}^{\frac{N-2}{2}}\de_{i_n}^{\frac{N-2}{2}}}{(\de_h^2+|\de_i y+\xi_{i_n}-\xi_{h_n}|^2)^\frac{N-2}{2}})}_{\O(\eps_n^{(\al-\hat\beta)(N-2)})=o_n(1)}\)\tilde{\phi}_{i_n} \varphi \\&+
					\int_{\Om_{i_n}} f'(\U_{1,0}) \tilde{\phi}_{i_n} \varphi + o_n(1)\\ \to & p\int_{\R} \U_{1,0}^{p-1} \varphi\tilde\phi_{i_\infty} 
			\end{align*}}  
			\normalcolor{}{where the first term goes to zero by the Dominated Convergence Theorem.} Hence at the limit \[-\Delta \tilde{\phi}_{i_\infty} = p\U_{1,0}^{p-1}\tilde{\phi}_{i_\infty} \qquad \mbox{weakly in } \mathcal D^{1,2}(\R).\]
			We want to prove that the limit function $\tilde{\phi}_{i_\infty}$ is null. It is sufficient to show that \[\tilde{\phi}_{i_\infty}\in \ker(-\Delta-pU^{p-1})^\perp , \] i.e.
			\[\int_{\R}\nabla\tilde{\phi}_{i_\infty}\cdot \nabla\psi^j= \int_{\R} \tilde{\phi}_{i_\infty}f'(\U_{1,0})\psi^j = 0 \mbox{ for all }j=0,\cdots,N. \]
			Indeed, if $j=0$ we have that
			\begin{align*}
				0=\l\phi_n,P\psi_{i_n}^j\r&= \int_\Om \nabla \phi_{n}(x)\cdot \nabla P\psi_{i_n}^j dx = \int_\Om \phi_n(x) f'(U_i(x))\psi_i^j(x) dx \\&= \int_{\Om_{in}} \tilde{\phi}_{i_n}(y)  f'(\U_{1,0}(y))\psi^0(y) dy  \to \int_{\R} \tilde{\phi}_{i_\infty}f'(\U_{1,0})\psi^0
			\end{align*} and similarly for $j=1,\cdots,N$.  
			\normalcolor{}{ As $\tilde\phi_{i_n}\rightharpoonup\tilde\phi_{i_\infty}$ in $\mathcal D^{1,2}(\R)$ and $f'(U_{1,0})\psi^0_{1,0}\in L^{\frac{2N}{N+2}}(\RR^N)$ then $\tilde\phi_{i_\infty}\rightharpoonup 0$ in $\mathcal D^{1,2}(\R)$.} 
			\item[$\cdot$] Step 3. 
			Now for all $\varphi \in C_c^\infty(\Om)$ we have that \normalcolor{}{
				\begin{align*}
					\l {\phi_{n}},\varphi \r_{H_0^1(\Om)} &=  \int_\Om f'\left(u_0-\sum \U_{h_n}\right){\phi_{n}} \varphi + \l \underbrace{h_n+w_n, \varphi\r}_{=o_n(1)} \\&= 		\int_{\Om}f'\left(u_0\right)\phi_n\varphi + \int_{\Om } \underbrace{\left[ f'(u_0-\sum \U_{h_n})-f'(u_0)\right]}_{\leq f'(\sum \U_{h_n})} \phi_n\varphi +o_n(1) 
					\\&\leq \int_\Om f'(u_0)\phi_n\varphi + \int_\Om \left[f'\(\sum \U_{h_n}\)-f'\(\U_{1_n}\)\right]|\phi_n||\varphi|+ \int_\Om f'\(\U_{1_n}\)|\phi_n||\varphi|+ o_n(1) \\&\leq \int_\Om f'(u_0)\phi_n\varphi + \sum \int_\Om f'\(\U_{i_n}\)|\phi_n||\varphi|+ o_n(1)
					\\&= \int_\Om f'(u_0)\phi_n\varphi + \sum \de_{i_n}^\frac{N+2}{2}\int_{\Om_{i_n}}f'\(\sum \U_{h_n}(\de_{i_n}y+\xi_{i_n})\)|\tilde \phi_{i_n}||\varphi|+ o_n(1)
					\\&= \int_\Om f'(u_0)\phi_n\varphi + \sum \de_{i_n}^\frac{N-2}{2}\int_{\Om_{i_n}}f'\(U_{1,0}\)|\tilde \phi_{i_n}||\varphi|+ o_n(1)
					\\&  \to p\int_\Om u_0^{p-1} \varphi\phi_{\infty} 
			\end{align*} }
			as $\tilde\phi_{i_n}$ converges to zero weakly in $\D^{1,2}(\R)$ and $ \U_{1,0}\in L^\frac{2N}{N+2}(\R)$.
			Hence $\phi_\infty$ weakly satisfies $-\Delta\phi_\infty = p u_0^{p-1}\phi_\infty$ in $\Om$ and, by the non degeneracy of $u_0$, we have that $\phi_\infty=0$.
			
			\item[$\cdot$] Step 4. Now we want to show that $\phi_n\to 0$ strongly in $H_0^1(\Om)$. 		
			Indeed
			\begin{align*}
				\|\phi_n\|_{H_0^1(\Om)} &= \normalcolor{}\l \i(f'\(u_0-\sum \U_{h_n}\)\phi_n,\phi_n \r + \underbrace{\l h_n+w_n,\phi_n \r}_{=0}\\&= \int_\Om \(f'\(u_0-\sum \U_{h_n}\) -f'\(\sum \U_{h_n}\)\) \phi_n^2+\int_\Om f'\(\sum \U_{h_n}\) \phi_n^2 \\&\leq  \int_\Om f'\(u_0\) \phi_n^2+\int_{\Om_{i_n}} f'\(U_{1,0}+\de_{i_n}^\frac{N-2}{2}\sum_{h_n\neq i_n} \U_{h_n}(\de_{i_n}y+\xi_{i_n})\) \tilde\phi_{i_n}^2 \\&\leq \int_\Om f'(u_0)\phi_n^2+\int_{\Om_{i_n}} f'(U_{1,0}) \tilde\phi_{i_n}^2 + \sum_{h\neq i } \de_{i_n}^2\int_{\Om_{i_n}} \al_N^\frac{4}{N-2}\( \frac{\de_{h_n}}{\de_{h_n}^2+|\de_{i_n}y+\xi_{i_n}-\xi_{h_n}|^2}\)^2\tilde\phi_{i_n}^2 
			\end{align*}
			where \normalcolor{}{the last term converges to zero by the Dominated Convergence Theorem.} Moreover \normalcolor{}{ $\phi_n^2$ and $\tilde\phi^2_{h_n}$ are uniformly bounded respectively in $L^{\frac{N}{N-2}}(\Om)$ and $L^{\frac{N}{N-2}}(\RR^N)$ and they converge to zero almost everywhere in $\Om$ and $\R$. Hence $\phi_n^2$ and $\tilde\phi_{h_n}^2$ converge weakly to zero in $L^\frac{N}{N-2}(\Om)$ and $L^\frac{N}{N-2}(\R)$ }. As $f'(u_0)\in L^\frac{N}{2}(\Om)$ and $  f'\(U_{1,0}\) \in L^\frac{N}{2}(\R)$, we have that $\|\phi_n\|_\H~=~o_n(1)$, and this contradicts the hypothesis  $\|\phi_n\|=1$ and prove that $\|\L_{\bm\de,\bm\xi}\phi\|~\geq~C~\|\phi\|$.
			\item[$\cdot$] Step 5. Now we have to prove that $\L_{\bm\de,\bm\xi}$ is invertible and its inverse is continuous. Indeed, we know that $\Pi^\perp\circ \i:L^{\frac{2N}{N-2}}(\Om)\to H_0^1(\Om)$ is a {compact} operator, then $\L_{\bm\de,\bm\xi}=Id-K$ where $K$ is a compact operator. By \eqref{L in} we also know that $\L_{\bm\de,\bm\xi}$ is injective. Thus,
			by the Fredholm’s alternative theorem is also surjective.
		\end{itemize}
	\end{proof}
	
	Now we want to estimate the error term defined in \eqref{err}.
	\begin{prop} \label{estimate err}
		Let $N>6$ and $\bm\de,\bm\xi$ satisfy \eqref{deltai} and \eqref{par} {as in Section \ref{ans}}. Then 
		there exists $C= C(a)>0$ and $\eps_0>0$ such that for any $\eps \in (0,\eps_0)$ 
		the error term satisfies
		\begin{equation*}
			\|\E_{\bm\de,\bm\xi}\|_{H^1_0(\Om)} \leq C \eps^{\frac{\hat\theta}{2}+\sigma} 
		\end{equation*} for some $\sigma>0$ arbitrarily small, where $\hat\theta$ is given by \eqref{htheta}.
	\end{prop}
	\begin{proof}
		Observing that
		$$	\E_{\bm\de,\bm\xi} = \Pi^\perp\{\i\left[f(\Wd)+\eps\Wd\right] -\i [f(u_0)-\sum_{h=1}^k f(\U_h)] \} $$ and by \eqref{ibound} we get
		\begin{align*}
			\|\E_{\bm\de,\bm\xi}\|_{H^1_0(\Om)} &\leq  |f(\Wd)-f(u_0)+\sum f(P\Uh)+\eps(\Wd)|_{\frac{2N}{N+2},\Om} \\&\leq 
			{\underbrace{\big|f(\Wd)-f\(\sum P\Uh\)\big|_{\frac{2N}{N+2};\cup B_{\eta_i}(\xi_i)}}_{(I)} +\sum_h \underbrace{|f(\U_h)|_{\frac{2N}{N+2}; \Om \backslash \cup B_{\eta_i}(\xi_i)}}_{(II)} }\\&+\underbrace{\big|\sum f(\Uh)-f\(\sum P\Uh\)\big|_{\frac{2N}{N+2},\cup B_{\eta_i}(\xi_i)}}_{(III)}\\& + \underbrace{|f(\Wd)-f(u_0)|_{\frac{2N}{N+2};\Om \backslash\cup B_{\eta_i}(\xi_i)}}_{(IV)}+ \underbrace{|f(u_0)|_{\frac{2N}{N+2}; \cup B_{\eta_i}(\xi_i)}}_{(V)}\\&+ \underbrace{\eps |u_0|_{\frac{2N}{N+2}}}_{(VI)}+\underbrace{\eps\sum_h |P\Uh|_{\frac{2N}{N+2};\Om}}_{(VII)}
			.
		\end{align*}
		Recalling that $\eta_h=\dist(\xi_h,\partial\Om)= \O(\eps^\beta)$, we observe that 
		$$ \ba
		(V) &\leq \sum |u_0^p|_{\frac{2N}{N+2}, B_{\eta_i}(\xi_i)} \leq  \sum_i \(\int_0^{\eta_i} r^{N-1} \)^{\frac{N+2}{2N}} \\&= \sum\O\(\eta_i^{\frac{N+2}{2}}\)= \O\(\eps^{\frac{N^2-4}{N^2-6N+4}}\)= \O\(\eps^{\frac{\hat\theta}{2}+\sigma}\)
		\ea $$
		and that \normalcolor{}{\beq \U_{\de_h,\xi_h} \lesssim   \(\frac{\de_h}{\eta_h^2}\)^{\frac{N-2}{2}} = \O\(\eps^{(\alpha-2\beta)\frac{N-2}{2}}\)= \O\(\eps^{\frac{N-2}{N^2-6N+4}}\) \label{U out}\eeq  in $\Om \backslash B_{\eta_h}(\xi_h)$ for all $h=1,\cdots, k$.} Then for all $N\geq 7$
		\begin{align*}
			(I) &\leq \sum_i
			\Bigg|f(u_0-\sum_h P\Uh) -  f\(\sum_h P\Uh\) \Bigg|_{\frac{2N}{N+2};B_{\eta_i}(\xi_i)} \\&\leq \sum_i \( \Bigg|f'\(\sum_h P\Uh\)u_0 \Bigg|_{\frac{2N}{N+2};B_{\eta_i}(\xi_i)}+ | u_0^p  |_{\frac{2N}{N+2};B_{\eta_i}(\xi_i)}  \)
			\\& \leq C \sum_i \(\Bigg|f'\(\sum_h P\Uh\)- f'\( \Ui\)\Bigg|_{\frac{2N}{N+2};B_{\eta_1}(\xi_1)} + |f'\(\Ui\) |_{\frac{2N}{N+2};B_{\eta_1}(\xi_1)}+ \(\int_0^{\eta_i} r^{N-1}\)^{\frac{N+2}{2N}}\) 
			\\&\leq C \sum_i \(\big|f'\(\Ui-P\Ui\)\big|_{\frac{2N}{N+2};B_{\eta_i}(\xi_i)} + \sum_{h\neq i}|f'\(P\Uh\) |_{\frac{2N}{N+2};B_{\eta_i}(\xi_i)}\)\\&+C\sum_i\(|f'\(\Ui\) |_{\frac{2N}{N+2};B_{\eta_i}(\xi_i)}+ \eta_i^{\frac{N+2}{2}}\)
			\\ &\leq C\sum_i \(|\varphi_{\de_i,\xi_i}|_{\infty,\Om}^{\frac{4}{N-2}}\eta_i^{\frac{N+2}{2}}+\eta_i^\frac{N+2}{2}\sum_{h\neq i} \frac{\de_h^2}{\eta_h^4} + \de_i^2 \(\int_0^{\eta_i} r^{N\frac{N-6}{N+2}-1}\)^\frac{N+2}{2}+\eta_i^\frac{N+2}{2}\)
			\\&\leq C \sum_i \frac{\de_i^2}{\eta_i^4}\eta_i^\frac{N+2}{2}+ +\eta_i^\frac{N+2}{2}\sum_{h\neq i} \frac{\de_h^2}{\eta_h^4} + \de_i^2 \eta_i^{\frac{N-6}{2}}+\eta_i^\frac{N+2}{2} = \O(\eps^{\frac{\hat\theta}{2}+\sigma}) 
		\end{align*} by \eqref{U out} and \eqref{varphi infty}.
		Now
		\beq \ba\label{fUn}		(II) &\leq  |f(\Uh)|_{\frac{2N}{N+2},\Om \backslash  B_{\eta_h}(\xi_h)}=\alpha_N^p \(\int_{\Om\backslash B_{\eta_h}(0)}  \(\frac{\de_h}{\de_h^2+|y|^2}\)^N\)^{\frac{N+2}{2N}} \\&\leq C \(\int_{\frac{\eta_h}{\de_h}}^{+\infty}  \frac{r^{N-1}}{(1+r^2)^{N}}\)^{\frac{N+2}{2N}} \leq C\(\frac{\de_h}{\eta_h}\)^{\frac{N+2}{2}}\\&= \O \(\eps^{\frac{N+2}{2}(\alpha-\beta)}\) = \O\(\eps^{\frac{N(N+2)}{N^2-6N+4}}\) =\O\(\eps^{\frac{\hat\theta}{2}+\sigma}\)
		\ea\eeq
		and analogously
		\beq \ba\label{Un}
		|\Uh|_{\frac{2N}{N+2},\Om \backslash \cup B_{\eta_i}(\xi_i)} &\leq  |\Uh|_{\frac{2N}{N+2},\Om \backslash  B_{\eta_h}(\xi_h)}= \alpha_N\(\int_{\Om\backslash B_{\eta_h}(0)} \(\frac{\de_h}{\de_h^2+|y|^2}\)^{N\frac{N-2}{N+2}}\)^{\frac{N+2}{2N}} \\&\leq C \de_h^2\(\int_{\frac{\eta_h }{\de_h}}^{+\infty} \alpha_N \frac{r^{N-1}}{(1+r^2)^{N\frac{N-2}{N+2}}}\)^{\frac{N+2}{2N}} \leq C\de_h^2 \(\frac{\de_h}{\eta_h}\)^{\frac{N-6}{2}}\\&= \O \(\eps^{2\alpha+\frac{N-6}{2}(\alpha-\beta)}\) = \O\(\eps^{\frac{N^2+2N-8}{2(N^2-6N+4)}}\) =\O\(\eps^{\frac{\hat\theta}{2}+\sigma}\).
		\ea\eeq

		Then, by \eqref{fUn} and \eqref{Un}
		\begin{align*}
			(IV)&\leq \big|\(\sum P\Uh\)^p|_{\frac{2N}{N+2};\Om \backslash\cup B_{\eta_i}(\xi_i)}+ |u_0^{p-1}\sum P\Uh|_{\frac{2N}{N+2};\Om \backslash\cup B_{\eta_i}(\xi_i)}
			\\&\leq C \sum|\(P\Uh\)^p|_{\frac{2N}{N+2};\Om \backslash\cup B_{\eta_i}(\xi_i)}+ C \sum|P\Uh|_{\frac{2N}{N+2};\Om \backslash\cup B_{\eta_i}(\xi_i)}
			\\&\leq C \sum|\(\Uh\)^p|_{\frac{2N}{N+2};\Om \backslash\cup B_{\eta_i}(\xi_i)}+ C \sum|\Uh|_{\frac{2N}{N+2};\Om \backslash \cup B_{\eta_i}(\xi_i)}= \O(\eps^{\frac{\hat\theta}{2}+\sigma}).
		\end{align*}
		In a similar way, by \eqref{varphi infty} and \eqref{U out} 
		\begin{align*}
			(III) \leq& \sum_i \Bigg|\sum_h f(\Uh)-f\(\sum_h P\Uh\)\Bigg|_{\frac{2N}{N+2},B_{\eta_i}(\xi_i)} \\
			\leq& \sum_i \( |f(\Ui)-f(\sum P\Uh)|_{\frac{2N}{N+2},B_{\eta_i}(\xi_i)} + \sum_{h\neq i} |f(\Uh)|_{\frac{2N}{N+2},B_{\eta_i}(\xi_i)} \)\\ 
			\leq& C \sum_i \(|f'(\Ui)\(\varphi_i+\sum_{h\neq i} P\Uh\) |_{\frac{2N}{N+2},B_{\eta_i}(\xi_i)}+ \sum_{h\neq i} |f(\Uh)|_{\frac{2N}{N+2},B_{\eta_i}(\xi_i)}\)
			\\ \leq& C \sum_i \( |f'(\Ui)|_{\frac{2N}{N+2},B_{\eta_i}(\xi_i)} \(|\varphi_{\de_i,\xi_i}|_{\infty,\Om} + \sum_{h\neq i}\(\frac{\de_h}{\eta_h^2}\)^\frac{N-2}{2} \)+ |B_{\eta_i}(\xi_i)|^{\frac{N+2}{2N}} \sum_{h\neq i} \(\frac{\de_h}{\eta_h^2}\)^\frac{N-2}{2}\)  \\
			\leq & C \sum_i  \( \de_i^{\frac{N-2}{2}} \( \int_0^{\frac{\eta_h}{\de_h}} r^{\frac{N^2-7N-2}{N+2}}\)^{\frac{N+2}{2N}}\) \sum \(\frac{\de_h^\frac{N-2}{2}}{\eta_h^{N-2}}\) + \eta_i^\frac{N+2}{2} \sum_{h\neq i} \(\frac{\de_h}{\eta_h^2}\)^\frac{N-2}{2} \\
			\leq & C \sum_i \( \de_i^{\frac{N-2}{2}} \( \frac{\eta_i}{\de_i}\)^{\frac{N-6}{2}} \sum \(\frac{\de_h^\frac{N-2}{2}}{\eta_h^{N-2}}\) + \eta_i^\frac{N+2}{2} \sum_{h\neq i} \(\frac{\de_h}{\eta_h^2}\)^\frac{N-2}{2} \) \\
			=& \O\(\eps^{(\alpha-\beta) \frac{N+2}{2}}\) + \O\(\eps^{\alpha \frac{N-2}{2}-\beta\frac{N-6}{2}}\) = \O\(\eps^{\hat\theta+\sigma}\) . 
		\end{align*}
	\end{proof}
	
	Now we can solve $\eqref{eq piperp}$ using a fixed point argument.
	\begin{prop}\label{contraction}Let $N>6$ and $\bm\de,\bm\xi$
 satisfy \eqref{deltai} and \eqref{par} {as in Section \ref{ans}}. Then
		there exists $C= C(a)>0$ and $\eps_0>0$ such that for any $\eps \in (0,\eps_0)$ 
		there exists a unique $\phi_{\bm\de,\bm\xi} \in K_{\bm\de,\bm\xi}^\perp$
		solving 
  \begin{equation}\label{problem}\L_{\bm\de,\bm\xi}\phi_{\bm\de,\bm\xi} = \E_{\bm\de,\bm\xi}+\N^1\phi_{\bm\de,\bm\xi}+\N^2\phi_{\bm\de,\bm\xi}
\end{equation}	
   and satisfying 
		\begin{equation} \label{phi est}
			\|\phi_{\bm\de,\bm\xi}\|_\H \leq C \eps^{\frac{\hat\theta}{2}+\sigma} 
		\end{equation}
		for some $\sigma>0$.
		
	\end{prop}
	%
	\begin{proof}
		We want to show that $\mathcal T_{\bm\de,\bm\xi}:K_{\bm\de,\bm\xi}^\perp\to K_{\bm\de,\bm\xi}^\perp$ is a contraction where$$ \mathcal T_{\bm\de,\bm\xi}\phi = \L_{\bm\de,\bm\xi}^{-1}[\E_{\bm\de,\bm\xi}+\N^1\phi+\N^2\phi] \mbox{ for all } \phi\in K_{\bm\de,\bm\xi}^\perp $$	and the operators $\L_{\bm\de,\bm\xi}$, $\E_{\bm\de,\bm\xi}$, $\N^1$ and $\N^2$  are defined \eqref{L def}, \eqref{err}, \eqref{N1} and \eqref{N2}.

		Let $r>0$ and $\mathcal B_{\bm\de,\bm\xi}=\{\phi\in K_{\bm\de,\bm\xi}^\perp : \|\phi\|_\H \leq r\eps^{\frac{\hat\theta}{2}+\sigma}\}$.
		We observe that 
		\begin{align*}
			\|\N^1\phi\|_\H\leq &C \|\i [f(\Wd+\phi)-f(\Wd)-f'(\Wd)\phi+\eps\phi]\|_\H \\ \leq& C | f(\Wd+\phi)-f(\Wd)-f'(\Wd)\phi|_{\frac{2N}{N+2},\Om} +\eps|\phi|_{\frac{2N}{N+2},\Om} \\ \leq&C |f(\phi)|_{\frac{2N}{N+2},\Om}+ \eps \|\phi\|_\H  \leq C \|\phi\|_\H^p+C\eps\|\phi\|_\H  
		\end{align*} and 
		\begin{align*}
			\|\N^2\phi\|_\H\leq &C \Bigg|\Bigg|\i \left[\(f'\(u_0-\sum_{h=1}^kP\Uh\)-f'\(u_0-\sum_{h=1}^k\Uh\)\)\phi\right]\Bigg|\Bigg|_\H \\\leq &C \Bigg| \(f'\(u_0-\sum_{h=1}^kP\Uh\)-f'\(u_0-\sum_{h=1}^k\Uh\)\)\phi\Bigg|_{\frac{2N}{N+2},\Om} \\ \leq&   C \Big| f'\(\sum\varphi_{\de_h,\xi_h}\)\Big|_{\frac{N}{2},\Om} |\phi|_{\frac{2N}{N-2},\Omega} \leq  C \(\sum|\varphi_{\de_h,\xi_h}|_{\frac{2N}{N-2},\Om}\)^\frac{4}{N-2} \|\phi\|_\H.
		\end{align*}
		Hence by \eqref{varphi 2star} \[\|\mathcal T_{\bm\de,\bm\xi}\phi\|_\H \leq \|\N^1\phi\|_\H+\|\N^2\phi\|_\H \leq C\eps^{\frac{\hat\theta}{2}+\sigma} \] and so $\mathcal T_{\bm\de,\bm\xi}$ maps $\mathcal B_{\bm\de,\bm\xi}$ into itself.
		Moreover for all $\phi_1,\phi_2\in K_{\bm\de,\bm\xi}^\perp$ we have
		\begin{align*}
			\|\N^1\phi_1-\N^1\phi_2\|_\H\leq &C |f({\Wd+\phi_1})-f({\Wd+\phi_2})-f'(\Wd)(\phi_1-\phi_2)|_{\frac{2N}{N+2},\Om} \\&+ \eps|\phi_1-\phi_2|_{\frac{2N}{N+2},\Om} \\ \leq& |f(\underbrace{\Wd+\phi_1}_{=a+b})-f(\underbrace{\Wd+\phi_2}_{=a;\  b=\phi_1-\phi_2})-f'(\Wd+\phi_2)(\phi_1-\phi_2)|_{\frac{2N}{N+2},\Om} \\&|f'(\Wd+\phi_2)(\phi_1-\phi_2)-f'(\Wd)(\phi_1-\phi_2)|_{\frac{2N}{N+2},\Om}+ \eps|\phi_1-\phi_2|_{\frac{2N}{N+2},\Om}\\ \leq
			&C\( ||\phi_1-\phi_2|^p|_{\frac{2N}{N+2},\Om}+ ||\phi_2|^p|\phi_1-\phi_2|_{\frac{2N}{N+2},\Om}+ \eps|\phi_1-\phi_2|_{\frac{2N}{N+2},\Om}\)\\ \leq& C\( |(\phi_1-\phi_2)^{p-1}|_{\frac{N}{2},\Om}|\phi_1-\phi_2|_{\frac{2N}{N-2},\Om}+ |\phi_2^{p-1}|_{\frac{N}{2},\Om}|\phi_1-\phi_2|_{\frac{2N}{N-2},\Om}\) \\&+C\eps \|\phi_1-\phi_2\|_\H^{p-1} \\ \leq&  C\( \underbrace{\|\phi_1-\phi_2\|_\H}_{\lesssim\|\phi_1\|_\H+\|\phi_2\|_\H}\|\phi_1-\phi_2\|_\H+ \|\phi_2\|_\H^{p-1}\|\phi_1-\phi_2\|_\H\)\\&+ C\eps \|\phi_1-\phi_2\|_\H\\ \leq&  C \(\|\phi_1\|_\H^{p-1}+ \|\phi_2\|_\H^{p-1}+ \eps \)\|\phi_1-\phi_2\|_\H
		\end{align*} and
		\begin{align*}
			\|\N^2\phi_1-\N^2\phi_2\|_\H\leq &C |[f'({u_0-\sum P\Uh})-f'({u_0-\sum \Uh})](\phi_1-\phi_2)|_{\frac{2N}{N+2},\Om} \\ \leq & C|f'(\sum \varphi_{\de_h,\xi_h})|_{\frac{N}{2},\Om} |\phi_1-\phi_2|_{\frac{2N}{N-2},\Om} \\ \leq& C |\sum \varphi_{\de_h,\xi_h}|^{p-1}_{\frac{2N}{N-2},\Om}\|\phi_1-\phi_2\|_\H \\ \leq & C \(\sum |\varphi_{\de_h,\xi_h}|_{\frac{2N}{N-2},\Om}\)^{p-1}\|\phi_1-\phi_2\|_\H \\ \leq &C \sum |\varphi_{\de_h,\xi_h}|_{\frac{2N}{N-2},\Om}^{p-1} \|\phi_1-\phi_2\|_\H .
		\end{align*}
		Hence \begin{align*}
			\|\mathcal T_{\bm\de,\bm\xi}\phi_1-\mathcal T_{\bm\de,\bm\xi}\phi_2\|_\H \leq &C \|\N^1\phi_1-\N^1\phi_2\|_\H + \|\N^2\phi_1-\N^2\phi_2\|_\H\\ \leq &C \(\|\phi_1\|_\H+\|\phi_2\|_\H+\eps + \sum |\varphi_{\de_h,\xi_h}|_{\frac{2N}{N-2},\Om}^{p-1}\) \|\phi_1-\phi_2\|_\H 
		\end{align*} where $|\varphi_{\de_h,\xi_h}|_{\frac{2N}{N-2},\Om}^{p-1}=\O(\eps^{2(\alpha-\beta)})$. Then, for $\eps$ sufficiently small, $\mathcal T_{\bm\de,\bm\xi}$ is a contraction and the claim follows.
	\end{proof}

	\section{The reduced energy} 
	In this section we want to reduce the original problem to a finite dimensional one and solve \eqref{eq pi}. 
	Define $J_\eps:\H\to\R$ as \beq\label{funz}
	J_\e(u)=\frac 12 \int_\Omega |\nabla u|^2 \, dx -\frac{1}{p+1}\int_\Omega |u|^{p+1}\, dx-\frac\e 2 \int_\Omega u^2\, dx. \eeq  Then the critical points of $J_\eps$ are solutions to \eqref{BN}. Let us also define the reduced functional \begin{equation}\label{Jtilde}
		\tilde J_\eps (\bm t,\bm d, \bm\tau) = J_\eps(\Wd+\phi_{\bm\de,\bm\xi})
	\end{equation} where $\bm t=(t_1,\cdots,t_k)$, $\bm d = (d_1,\cdots, d_k)\in \RR^k$, $\bm\tau = (\tau_1,\cdots,\tau_k)\in (\RR^{N-1})^k$ as in Section \ref{ans},  $\Wd=u_0-\sum P\Uh$ and $\phi_{\bm\de,\bm\xi}$ is as in Proposition \ref{contraction} such that \[
		\|\phi_{\bm\de,\bm\xi}\|_\H \leq C\eps^{\frac{\hat\theta}{2}+\sigma}
	\] for some $\sigma>0$ and $\hat\theta=\frac{N^3-8N+8}{N(N^2-6N+4)}$.
	
	\begin{prop} 
		Let $N>6$ and $\bm\de,\bm\xi$ satisfy \eqref{deltai} and \eqref{par} {as in Section \ref{ans}}. There exists $\eps_0>0$ such that for any $\eps \in (0,\eps_0)$ it holds
	$$
			J_\eps(\Wd+\phi_{\bm\de,\bm\xi})=J_\eps(\Wd)+\O(\eps^{\hat\theta+\sigma}),
		$$
  where $J_\eps$ is defined in \eqref{funz}.
	\end{prop}	
	\begin{proof}
		It's easy to see that 
		\begin{align*}
			J_\eps(\Wd+\phi_{\bm\de,\bm\xi})-J_\eps(\Wd) =& \frac 12 \|\phi_{\bm\de,\bm\xi}\|_\H^2 + \int_\Om \nabla\Wd \cdot\nabla \phi_{\bm\de,\bm\xi} - \int_\Om \Wd^p\phi_{\bm\de,\bm\xi} \\&- \frac{1}{p+1} \int_\Om \left[|\Wd+\phi_{\bm\de,\bm\xi}|^{p+1}+|\Wd|^{p+1}-(p+1)\Wd^p\phi_{\bm\de,\bm\xi}\right] \\& -\frac \eps 2\int_\Om \phi_{\bm\de,\bm\xi}^2-\eps\int_\Om \Wd\phi_{\bm\de,\bm\xi} .
		\end{align*}
		Using \eqref{phi est} 
		\begin{align*}
			\eps \int_\Om \phi_{\bm\de,\bm\xi}^2 \leq \eps |\phi|_{2,\Om} \leq \eps |\phi_{\bm\de,\bm\xi}|_{\frac{2N}{N-2},\Om}^2\leq \eps \|\phi_{\bm\de,\bm\xi}\|_\H^2 \leq C \eps^{\hat\theta+\sigma}
		\end{align*} and
		\begin{align*}
			\eps \int_\Om \Wd\phi_{\bm\de,\bm\xi} \leq \eps |\Wd|_\frac{2N}{N+2}|\phi_{\bm\de,\bm\xi}|_{\frac{2N}{N-2},\Om} \leq C \eps^{1+2\alpha}\|\phi_{\bm\de,\bm\xi}\|_\H \leq C \eps^{\hat\theta+\sigma}
		\end{align*} by \eqref{est Uconj}.
		Now by \eqref{est U2star} and \eqref{phi est}
		\begin{align*}\int_\Om \left[|\Wd+\phi_{\bm\de,\bm\xi}|^{p+1}+|\Wd|^{p+1}\right. &\left.-(p+1)\Wd^p\phi_{\bm\de,\bm\xi}\right] \leq \int_\Om  \left[|\phi_{\bm\de,\bm\xi}|^{p+1}+\phi^2\Wd^{p-1}\right] \\ &\leq |\phi_{\bm\de,\bm\xi}^{p+1}|_{\frac{2N}{N+2},\Om}+|\phi_{\bm\de,\bm\xi}|_{\frac{N}{N-2},\Om}|\Wd^{p-1}|_{\frac{N}{2},\Om} \\ \leq & C \|\phi_{{\bm\de,\bm\xi}}\|_\H^{p+1}+\|\phi_{\bm\de,\bm\xi}\|_\H^2|\Wd|_{\frac{2N}{N-2},\Om}^{p-1} \\ \leq & C \(\eps^{(p+1)\frac{\hat\theta}{2}+\sigma}+\eps^{{\hat\theta}+\sigma} \)\leq C \eps^{{\hat\theta}+\sigma}.
		\end{align*}
		Moreover
		\begin{align*}
			\int_\Om \nabla\Wd \cdot\nabla \phi_{\bm\de,\bm\xi} - \int_\Om \Wd^p\phi_{\bm\de,\bm\xi} =& -\int_\Om \left[\Delta \Wd-\Wd^p\right]\phi_{\bm\de,\bm\xi} \\=& \int_\Om  \left[f(u_0)-\sum f(\Uh)-f(u_0-\sum P\Uh )\right]\phi_{\bm\de,\bm\xi} \\ \leq & |f(u_0)-\sum f(\Uh)-f(u_0-\sum P\Uh )  |_{\frac{2N}{N+2},\Om} |\phi_{\bm\de,\bm\xi}|_{\frac{2N}{N-2},\Om} 
			\\\leq & C \eps^{\hat\theta+\sigma} 
		\end{align*} as in the proof of Proposition  \ref{estimate err}.
	\end{proof}
	
	Now we want to evaluate and find an expansion of $$J_\eps\(u_0-\sum_{i=1}^k P\Ui\).$$

	\begin{prop} \label{prop energy}
		Let $N>6$ and $\bm\de,\bm\xi$ satisfy \eqref{deltai} and \eqref{par} {as in Section \ref{ans}}. Then 
		it holds
		\begin{equation*}\begin{aligned}
				\tilde J_\eps (\bm d,\bm t,\bm \tau) 
				&=\e^{\hat\theta}\left(\mathfrak A \sum_i d_i^2+\mathfrak C\sum_i t_i^2 -\alpha_N\mathtt C \sum_{h<i}\frac{d_0^{N-2}}{|\tau_i-\tau_h|^{N-2}}-\frac{\mathtt C}{2}d_0^{\frac{N-2}{2}}t_0\sum_i \langle D_{N-1}^2\partial_\nu u_0(\xi_0)\tau_i, \tau_i\rangle\right)\\
				&+\mathtt A_1
				+\mathcal O\left(\e^{\hat\theta+\sigma}\right).\end{aligned}\end{equation*}
		$C^0$-uniformly with respect to  $\bm \tau=(\tau_1,\cdots,\tau_k) \in(\RR^{N-1})^k$, $\bm d=(d_1,\cdots,d_k), \bm t=(t_1,\cdots,t_k)\in \RR^{k}$ such that $|d_i|,|t_i|<a$ and $|\tau_i-\tau_h|>\rho$ for all $i=1,\cdots,k$ and $i\neq h$, and where
		\[\mathtt A_1 = \mathtt A+\e^{\theta} \mathfrak g(d_0, t_0)-k \e^{\theta+\sigma}\mathfrak f(d_0, t_0) . \]	
	\end{prop} 
  For the definition of $\tilde J_\eps$ we refer to \eqref{Jtilde}.
	\begin{proof} We can write $J_\e(u_0-\sum_{i=1}^k P\Ui)$ as
		$$\begin{aligned}
			&J_\e(u_0-\sum_{i=1}^k P\Ui)=\underbrace{\frac12\int_\Omega |\nabla u_0|^2-\frac{1}{2^\star}\int_\Omega|u_0|^{2^\star}}_{:=J_0(u_0)}-\e\underbrace{\frac 12 \int_\Omega u_0^2}_{:=\tilde J_0(u_0)}+\underbrace{\sum_{i=1}^k\left(\frac 12 \int_\Omega|\nabla P\Ui|^2-\frac{1}{2^\star}\int_\Omega (P\Ui)^{2^\star}\right)}_{:=(I)}\\
			&+\underbrace{\sum_{i>h}\int_\Omega\left(\nabla P\Ui\nabla P\Uh-f(P\Ui)P\Uh\right)}_{:=(II)}-\underbrace{\sum_{h<i}\int_\Omega f(P\Ui)P\Uh}_{:=(III)}\\
			&-\underbrace{\frac \e 2\sum_{i=1}^k\int_\Omega( P\Ui)^2}_{(IV)}-\underbrace{\e\sum_{h<i}\int_\Omega P\Ui P\Uh}_{:=(V)}-\underbrace{\e\sum_{i=1}^k\int_\Omega u_0 P\Ui}_{(VI)}\\
			&-\underbrace{\sum_{i=1}^k\int_\Omega u_0(P\Ui)^{2^\star-1}}_{:=(VII)}-\underbrace{\sum_{i=1}^k \int_\Omega \left(\nabla u_0\nabla P\Ui-f(u_0)P\Ui\right)}_{=0}\\
			&-\underbrace{\int_\Omega \left[F(u_0-\sum_{i=1}^k P\Ui)-F(u_0)-\sum_{i=1}^k F(P\Ui)-\sum_{i\neq h} f(P\Ui)P\Uh\right]}_{(VIII)}\\
			&-\underbrace{\int_\Omega\left[\sum_{i=1}^kf(u_0)P\Ui-\sum_{i=1}^ku_0 f(P\Ui)\right]}_{:=(IX)}
		\end{aligned}$$
		where  $f(u)=|u|^{p-1}u$ and $F(s)=\int_0^s f(t)dt$.
		Now we evaluate each term. Recalling that $\eta_i:={\rm dist}(\xi_i, \partial\Omega)=\O(\eps^\beta)$, we have
		$$\begin{aligned}
			(I)&=\left(\frac 12-\frac{1}{2^\star}\right)\sum_i \int_\Omega (\Ui)^{2^\star}+\frac 12 \sum_i \int_\Omega (\Ui)^{2^\star-1}\varphi_{\delta_i, \xi_i}+\mathcal O\left(\int_\Omega (\Ui)^{2^\star-2}\varphi_{\delta_i, \xi_i}^2+\int_\Omega \varphi_{\delta_i, \xi_i}^{2^\star}\right)\\
			&=\frac kN \int_{\mathbb R^N} \mathcal \U^{2^\star}+\frac 12 \alpha_N\sum_i \delta_i^{N-2}H(\xi_i, \xi_i)\int_{\mathbb R^N}\mathcal \U^{2^\star-1}+\mathcal O{\left(\frac{\delta^{N}_i}{\eta_i^N}\right)}
			\\		&=\frac kN \int_{\mathbb R^N} \mathcal \U^{2^\star}+\frac 12 \alpha_N\sum_i \delta_i^{N-2}H(\xi_i, \xi_i)\int_{\mathbb R^N}\mathcal \U^{2^\star-1}+\mathcal O\left(\e^{\hat\theta+\sigma}\right)
		\end{aligned}$$
		
		$$\begin{aligned}
			(IV)&=\frac \e 2 \sum_i \d_i^2\int_{\mathbb R^N}\mathcal U^2+\mathcal O(\e\sum_i\|\varphi_{\delta_i, \xi_i}\|_{L^2(\Omega)}^2)\\
			&+\e\sum_i\|\varphi_{\delta_i, \xi_i}\|_{L^\infty(\Omega)}\int_{B_{\eta_i}(\xi_i)}\Ui+\mathcal O(\e\sum_i\|\varphi_{\delta_i, \xi_i}\|_{L^\frac{2N}{N+2}(\Omega)}\|\Ui\|_{L^{2^\star}(\Omega\setminus B_{\eta_i}(\xi_i))})\\
			&=\frac \e 2 \sum_i \d_i^2\int_{\mathbb R^N}\mathcal U^2+\mathcal O\left(\e\frac{\delta^{N-2}_i}{\eta_i^{N-2}}\right)\\
			&=\frac \e 2 \sum_i \d_i^2\int_{\mathbb R^N}\mathcal U^2+\mathcal O\left(\e^{\hat\theta+\sigma}\right)\end{aligned}$$
		Now
		$$\begin{aligned}
			(VII)&=\sum_i \int_\Omega (u_0 \Ui)^{2^\star-1}+\sum_i \int_\Omega u_0\left((P\Ui)^{2^\star}-(\Ui)^{2^\star}\right)\\
			&=\sum_i \d_i^{\frac{N-2}{2}}\int_{\frac{\Omega-\xi_i}{\eta_i}}u_0(\delta_i y+\xi_i)U^{2^\star-1}+\mathcal O\left(\sum_i\int_{B_{\eta_i}(\xi_i)}u_0(x)\varphi_{\d_i, \xi_i}^{2^\star-1}\right)\\
			&+\mathcal O\left(\sum_i \int_{B_{\eta_i}(\xi_i)}u_0(x) (\Ui)^{2^{\star}-2}\varphi_{\d_i, \xi_i}\right)+\mathcal O\left(\sum_i \int_{\Omega\setminus B_{\eta_i}(\xi_i)}u_0(x)(\Ui)^{2^\star-1}\right) \\
			&=\sum_i \d_i^{\frac{N-2}{2}}u_0(\xi_i)\int_{\mathbb R^N} \U^{2^\star-1}+\mathcal O\left(\sum_i \d_i^{\frac{N-2}{2}}u_0(\xi_i)\int_{\frac{\eta_i}{\delta_i}}^{+\infty}\frac{r^{N-1}}{(1+r^2)^{\frac{N+2}{2}}}\right)\\
			&+\mathcal O\left(\sum_i \|\varphi_{\delta_i, \xi_i}\|^{2^\star-1}_{L^\infty(\Omega)}\eta_i^N u_0(\xi_i)\right)+\mathcal O\left(\sum_i \d_i^2 \|\varphi_{\delta_i, \xi_i}\|_{L^\infty(\Omega)}\eta_i^N u_0(\xi_i)\int_{B_1(0)}\frac{1}{|y|^4}\right)\\
			&+\mathcal O\left(\sum_i \d_i^{\frac{N-2}{2}}u_0(\xi_i)\int_{\frac{\eta_i}{\d_i}}^{+\infty}\frac{r^{N-1}}{(1+r^2)^{\frac{N+2}{2}}}\right)
		\end{aligned}$$
		Hence
		$$\begin{aligned}
			(VII)&=\sum_i \d_i^{\frac{N-2}{2}}u_0(\xi_i)\int_{\mathbb R^N} \U^{2^\star-1}+\mathcal O\left(\sum_i \d_i^{\frac{N-2}{2}}u_0(\xi_i)\frac{\d_i^2}{\eta_i^2}\right)\\
			&=\sum_i \d_i^{\frac{N-2}{2}}u_0(\xi_i)\int_{\mathbb R^N} \U^{2^\star-1}+\mathcal O\left(\e^{\hat\theta+\sigma}\right).
		\end{aligned}$$
		The other important term is $(III)$. Indeed
		$$\begin{aligned} (III)&=\sum_{h<i}\int_\Omega (P\Ui)^{2^\star-1}P\Uh=\sum_{h<i}\int_\Omega(\Ui)^{2^\star-1}\U_h-\sum_{h<i}\int_\Omega (\Ui)^{2^\star-1}\varphi_{\d_h, \xi_h}\\
			&+\sum_{h<i}\int_\Omega\left((P\Ui)^{2^\star-1}-(\Ui)^{2^\star-1}\right)P\Uh\\
			&=\sum_{h<i}\frac{\d_i^{\frac{N-2}{2}}\alpha_N\d_h^{\frac{N-2}{2}}}{|\xi_i-\xi_h|^{N-2}}\int_{\mathbb R^N}\U^{2^\star-1}\left(1+o(1)\right)+\mathcal O\left(\sum_{h<i}\frac{\d_h^{\frac{N-2}{2}}\d_i^{\frac{N-2}{2}}}{\eta_h^{N-2}}\right)\\
			&+\mathcal O\left(\sum_{h<i}\int_{B_{\eta_i}(\xi_i)}(\Ui)^{2^\star-2}\varphi_{\d_i, \xi_i}\Uh\right)+\mathcal O\left(\sum_{h<i}\int_{\Omega\setminus B_{\eta_i}(\xi_i)}(\Ui)^{2^\star-1}\Uh\right)\\
			&=\sum_{h<i}\frac{\d_i^{\frac{N-2}{2}}\alpha_N\d_h^{\frac{N-2}{2}}}{|\xi_i-\xi_h|^{N-2}}\int_{\mathbb R^N}\U^{2^\star-1}\left(1+o(1)\right)+\mathcal O\left(\sum_{h<i}\frac{\d_h^{\frac{N-2}{2}}\d_i^{\frac{N-2}{2}}}{\eta_h^{N-2}}\right)\\
			&+\mathcal O\left(\sum_{h<i}\frac{\d_i^{\frac{N-2}{2}}\d_h^{\frac{N-2}{2}}}{|\xi_h-\xi_h|^{N-2}}\frac{\d_i^2}{\eta_i^2}\right)\\
			&=\sum_{h<i}\frac{\d_i^{\frac{N-2}{2}}\alpha_N\d_h^{\frac{N-2}{2}}}{|\xi_i-\xi_h|^{N-2}}\int_{\mathbb R^N}\U^{2^\star-1}\left(1+o(1)\right)\\
			&=\sum_{h<i}\frac{\d_i^{\frac{N-2}{2}}\alpha_N\d_h^{\frac{N-2}{2}}}{|\xi_i-\xi_h|^{N-2}}\int_{\mathbb R^N}\U^{2^\star-1}+\mathcal O\left(\e^{\hat\theta+\sigma}\right)
		\end{aligned}$$
		We have to show that the other terms are of higher order.\\ Reasoning as the third term of (III) we have 
		$$\begin{aligned} (II)&=\sum_{h>i}\int_\Omega\left((P\Ui)^{2^\star-1}-(\Ui)^{2^\star-1}\right)P\Uh=\mathcal O\left(\sum_{h>i}\frac{\d_i^{\frac{N-2}{2}}\d_h^{\frac{N-2}{2}}}{|\xi_h-\xi_h|^{N-2}}\frac{\d_i^2}{\eta_i^2}\right)\\& =\mathcal O\left(\e^{\hat\theta+\sigma}\right)\end{aligned}$$
		Now
		$$\begin{aligned} (V)&\lesssim\e \sum_{i<h}\int_\Omega \Ui\Uh\lesssim \e \sum_{i<h} \d_i^{\frac{N-2}{2}}\d_h^{\frac{N-2}{2}}\int_{B_{\eta_i}(\xi_i)}\frac{1}{|x-\xi_i|^{N-2}|x-\xi_h|^{N-2}}\\ &+\e \sum_{i<h}\int_{\Omega\setminus B_{\eta_i}(\xi_i)} \Ui\Uh\\
			&\lesssim \e\sum_{i<h} \d_i^{\frac{N-2}{2}}\d_h^{\frac{N-2}{2}}\int_{B_{\eta_i}(0)}\frac{1}{|y|^{N-2}|y+\xi_i-\xi_h|^{N-2}}+\e\sum_{i<h} \frac{\d_i^{\frac{N-2}{2}}\d_h^{\frac{N-2}{2}}}{|\xi_i-\xi_h|^{N-2}}\d_i^2\int_{\frac{\eta_i}{\d_i}}^{+\infty}\frac{r^{N-1}}{(1+r^2)^{N-2}}\\
			&\lesssim \e\sum_{i<h} \frac{\d_i^{\frac{N-2}{2}}\d_h^{\frac{N-2}{2}}}{|\xi_i-\xi_h|^{N-2}}\int_{B_{\eta_i}(0)}\frac{1}{|y|^{N-2}}+\e\sum_{i<h} \frac{\d_i^{\frac{N-2}{2}}\d_h^{\frac{N-2}{2}}}{|\xi_i-\xi_h|^{N-2}}\eta_i^2\\
			&\lesssim \e \sum_{i<h} \frac{\d_i^{\frac{N-2}{2}}\d_h^{\frac{N-2}{2}}}{|\xi_i-\xi_h|^{N-2}}=o\left(\sum_{i<h} \frac{\d_i^{\frac{N-2}{2}}\d_h^{\frac{N-2}{2}}}{|\xi_i-\xi_h|^{N-2}}\right)=\mathcal O\left(\e^{\hat\theta+\sigma}\right)
		\end{aligned}$$
		Now
		$$\begin{aligned} (VI)&\lesssim\e \sum_i \d_i^{\frac{N-2}{2}}u_0(\xi_i) \eta_i^2=\mathcal O\left(\e^{\hat\theta+\sigma}\right) \end{aligned}$$
		Now we can split the term $(VIII)+(IX)$ into the sum of integrals on the balls $ B_{\eta_h}(\xi_h)$ and the integral on  $\Omega \backslash \cup B_{\eta_h}(\xi_h) $ that we call $A_h$ and $B$ respectively, 
		where	we can evaluate the integral $B$ on $\Omega \setminus \cup B_{\eta_h}(\xi_h)$ as
	$$\begin{aligned}|B|\leq &
		\int_{\Omega\setminus \bigcup_h B_{\eta_h}(\xi_h)}\left|F(u_0-\sum_h P\Uh)-F(u_0)+f(u_0)\sum_h P\Uh\right|\\
		&+\int_{\Omega\setminus \bigcup_h B_{\eta_h}(\xi_h)}u_0\sum_h f(P\Uh)+\sum_{i\neq h }\int_{\Omega\setminus \bigcup_h B_{\eta_h}(\xi_h)}f(P\Ui)P\Uh +\sum_h \int_{\Omega\setminus \bigcup_h B_{\eta_h}(\xi_h)}F(P\Uh)\\
		\lesssim & \sum_h \int_{\Omega\setminus \bigcup_h B_{\eta_h}}u_0^{2^\star-2}(\Uh)^2+\sum_h \int_{\Omega\setminus \bigcup_h B_{\eta_h}}u_0 (\Uh)^{2^\star-1}\\
		&+\sum_{i\neq h}\int_{\Omega\setminus \bigcup_h B_{\eta_h}}\Ui^{2^\star-1}\Uh+\sum_h \int_{\Omega\setminus \bigcup_h B_{\eta_h}}\Uh^{2^\star}\\
		\lesssim &\sum_h u_0^{2^\star-2}(\xi_h)\d_h^2\int_{\frac{\eta_h}{\d_h}}^{+\infty}\frac{r^{N-1}}{(1+r^2)^{N-2}}+\sum_h u_0(\xi_h)\d_h^{\frac{N-2}{2}}\int_{\frac{\eta_h}{\d_h}}^{+\infty}\frac{r^{N-1}}{(1+r^2)^{\frac{N+2}{2}}}\\
		&+\sum_{i\neq h} \frac{\d_i^{\frac{N-2}{2}}\d_h^{\frac{N-2}{2}}}{|\xi_i-\xi_h|^{N-2}}\int_{\frac{\eta_h}{\d_h}}^{+\infty}\frac{r^{N-1}}{(1+r^2)^{\frac{N+2}{2}}}+\sum_h \int_{\frac{\eta_h}{\d_h}}^{+\infty}\frac{r^{N-1}}{(1+r^2)^{N}}\\
		\lesssim &\sum_h u_0^{2^\star-2}(\xi_h)\d_h^2\left(\frac{\d_h}{\eta_h}\right)^{N-4}+\sum_h u_0(\xi_h)\d_h^{\frac{N-2}{2}}\left(\frac{\d_h}{\eta_h}\right)^2\\
		&+\sum_{i\neq h}\frac{\d_i^{\frac{N-2}{2}}\d_h^{\frac{N-2}{2}}}{|\xi_i-\xi_h|^{N-2}}\left(\frac{\d_h}{\eta_h}\right)^2+\sum_h \left(\frac{\d_h}{\eta_h}\right)^N\\
		=&\mathcal O\left(\e^{\hat\theta+\sigma}\right)
	\end{aligned}$$
	and in a similar way for the integral $A_h$  on a ball $B_{\eta_h}(\xi_h)$ we get that \[|A_h|= \O(\eps^{\hat\theta+\sigma}).\] 
	At the end we get
	
	$$\begin{aligned} &J_\e(u_0-\sum_h P\Ui)=\underbrace{J_0(u_0)-\e\tilde J_0(u_0)+\frac kN  \int_{\mathbb R^N}\mathcal U^{2^\star}}_{:=\mathtt A}\\
		&+\frac 12 \alpha_N \sum_i \d_i^{N-2}H(\xi_i, \xi_i)\int_{\mathbb R^N}\mathcal \U^{2^\star-1}-\frac \e 2 \sum_i \d_i^2\int_{\mathbb R^N}\mathcal \U^2-\sum_i \d_i^{\frac{N-2}{2}}u_0(\xi_i)\int_{\mathbb R^N} U^{2^\star-1}\\
		&-\sum_{h<i}\frac{\d_i^{\frac{N-2}{2}}\alpha_N\d_h^{\frac{N-2}{2}}}{|\xi_i-\xi_h|^{N-2}}\int_{\mathbb R^N}U^{2^\star-1}+\mathcal O\left(\e^{\hat\theta+\sigma}\right)\end{aligned}$$
	
	We put $$\mathtt B:=\frac 12\int_{\mathbb R^N}\mathcal U^2,\quad \mathtt C:=\int_{\mathbb R^N}\mathcal U^{2^\star-1}.$$
	
	Then
	$$\begin{aligned} &J_\e(u_0-\sum_i P\Ui)=\mathtt A+\frac 12 \alpha_N \mathtt C \sum_i \d_i^{N-2}H(\xi_i, \xi_i)- \e \mathtt B\sum_i \d_i^2-\mathtt C\sum_i \d_i^{\frac{N-2}{2}}u_0(\xi_i)\\
		&-\alpha_N\mathtt C\sum_{h<i}\frac{\d_i^{\frac{N-2}{2}}\d_h^{\frac{N-2}{2}}}{|\xi_i-\xi_h|^{N-2}}+\mathcal O\left(\e^{\hat\theta+\sigma}\right).\end{aligned}$$
	Now
	$$\begin{aligned}- \e \mathtt B\sum_h \d_i^2&=-\e\mathtt B\sum_i \left(\e^\alpha d_0+\e^{\hat\alpha}d_i\right)^2\\
		&=-\e^{1+2\alpha}\mathtt B k d_0^2-2\e^{1+\alpha+\hat\alpha}\mathtt B d_0\sum_i d_i -\e^{1+2\hat\alpha}\mathtt B\sum_i d_i^2 \end{aligned}$$
	Moreover
	$$\begin{aligned}
		&\frac 12 \alpha_N \mathtt C\sum_i \d_i^{N-2}H(\xi_i, \xi_i)=\frac 12 \alpha_N\mathtt C \sum_i \d_i^{N-2}\left[\frac{1}{2^{N-2}{\rm dist}(\xi_i, \partial\Omega)^{N-2}}+{o\left(\frac{1}{{\rm dist}(\xi_i, \partial\Omega)^{N-2}}\right)}\right]\\
		&=\frac{ 1}{2^{N-1}} \alpha_N\sum_i \e^{\alpha(N-2)}(d_0+\e^{\hat\alpha-\alpha})^{N-2}\left[\frac{1}{\e^{\beta(N-2)}t_0^{N-2}}-\frac{N-2}{\e^{\beta(N-1)}t_0^{N-1}}\e^{\tilde\beta}t_i+\frac{(N-2)(N-1)}{2\e^{\beta N}t_0^N}\e^{2\tilde\beta}t_i^2\right.\\
		&\left.+\mathcal O\left(\e^{3\tilde\beta-\beta(N+1)}\right)\right]+\normalcolor{}{o\left(\sum_i\frac{\de_i^{N-2}}{{\rm dist}(\xi_i, \partial\Omega)^{N-2}}\right)}\\
		&=\frac{\alpha_N}{2^{N-1}} \mathtt C\sum_i \e^{\alpha(N-2)}\left(d_0^{N-2}+(N-2)d_0^{N-3}d_i \e^{\hat\alpha-\alpha}+\frac{(N-2)(N-3)}{2}d_0^{N-4}d_i^2\e^{2\hat\alpha-2\alpha}+\mathcal O\left(\e^{3\hat\alpha-3\alpha}\right)\right)\\&*\left[\frac{1}{\e^{\beta(N-2)}t_0^{N-2}}-\frac{N-2}{\e^{\beta(N-1)}t_0^{N-1}}\e^{\tilde\beta}t_i+\frac{(N-2)(N-1)}{2\e^{\beta N}t_0^N}\e^{2\tilde\beta}t_i^2+\mathcal O\left(\e^{3\tilde\beta-\beta(N+1)}\right)\right]\\
		&+{o\left(\sum_i\frac{\de_i^{N-2}}{{\rm dist}(\xi_i, \partial\Omega)^{N-2}}\right)}\\
		&=\e^{(\alpha-\beta)(N-2)}\frac{\alpha_N}{2^{N-1}} \mathtt C k \frac{d_0^{N-2}}{t_0^{N-2}}\\
		&-\e^{\alpha(N-2)-\beta(N-1)+\tilde\beta}\frac{\alpha_N(N-2)}{2^{N-1}} \frac{d_0^{N-2}}{t_0^{N-1}}\mathtt C\sum_i t_i+\e^{\alpha(N-3)-\beta(N-2)+\hat\alpha}\frac{\alpha_N}{2^{N-1}} \frac{d_0^{N-3}}{t_0^{N-2}}\mathtt C\sum_id_i\\
		&+\e^{\alpha(N-2)-\beta N+2\tilde\beta}\frac{\alpha_N(N-2)(N-1)}{2^{N}} \frac{d_0^{N-2}}{t_0^{N}}\mathtt C\sum_i t_i^2-\e^{\alpha(N-3)+\hat\alpha-\beta(N-1)+\tilde\beta}\frac{\alpha_N(N-2)^2}{2^{N-1}} \frac{d_0^{N-3}}{t_0^{N-1}}\mathtt C\sum_i t_id_i\\
		&+\e^{\alpha(N-4)+2\hat\alpha-\beta(N-2)}\frac{\alpha_N(N-2)(N-3)}{2^{N}} \frac{d_0^{N-4}}{t_0^{N-2}}\mathtt C\sum_i d_i^2+\underbrace{\mathcal O\left(\e^{\alpha(N-2)+3\tilde\beta-\beta(N+1)}\right)}_{:=\mathcal O\left(\e^{\hat\theta+\sigma}\right)}\\
		&+{{o\left(\sum_i\frac{\de_i^{N-2}}{{\rm dist}(\xi_i, \partial\Omega)^{N-2}}\right)}}\\
	\end{aligned}$$
	Now by \eqref{diff xi} we have that
	$$-\alpha_N\mathtt C\sum_{h<i}\frac{\d_i^{\frac{N-2}{2}}\d_h^{\frac{N-2}{2}}}{|\xi_i-\xi_h|^{N-2}}=-\e^{\alpha(N-2)-\hat\beta(N-2)}\alpha_N \mathtt C\sum_{h<i}\frac{d_0^{N-2}}{|\tau_i-\tau_h|^{N-2}}(1+o(1)).$$
	At the end, observing that 
	\[   \hat\xi_i-\xi_0= (\eps^{\hat\beta}\tau_i, \vartheta(\xi'_0+\eps^{\hat\beta}\tau_i)-\vartheta(\xi_0'))=  (\eps^{\hat\beta}\tau_i, \eps^{\hat\beta} \underbrace{\nabla_{N-1}\vartheta(\xi'_0)^T}_{=0}\cdot\tau_i+\O(\eps^{2\hat\beta}))
	\] and 
\begin{align*}
		\nabla u_0(\hat\xi_i)^T\cdot\nu(\hat\xi_i) &= \partial_\nu u_0(\hat\xi_i) = \partial_\nu u_0(\xi_0)+\nabla_N \partial_\nu u_0(\xi_0)^T\cdot(\xi_i-\xi_0) \\&+ \frac{1}{2} (D_N^2\partial_{\nu}u_0(\xi_i) (\xi_i-\xi_0))^T\cdot(\xi_i-\xi_0)+ \O(|\hat\xi_i-\xi_0|^3)\\&=
		\partial_\nu u_0(\xi_0)+\eps^{\hat\beta} \underbrace{\nabla_{N-1}\partial_\nu u_0(\xi_0)^T}_{\mbox{\tiny We need this equal to zero}}\cdot \tau_i +\O(\eps^{2\hat\beta})\partial_{x_N}\partial_\nu u_0(\xi_0)\\&+ \frac 12 \eps^{2\hat\beta}(D_{N-1}^2\partial_\nu u_0(\xi_0)\tau_i)^T\cdot \tau_i + \O(\eps^{3\hat\beta}),
	\end{align*}
	as $\hat\xi_i\in \Om$ and $\xi_0$ is a critical point of $\partial_\nu u_0(\xi)$, we have
	\begin{align*}
		u_0(\xi_i)&= u_0(\hat\xi_i+(\eps^\beta t_0+\eps^{\tilde\beta}t_i)\nu(\hat\xi_i)) \\&= \underbrace{u_0(\hat\xi_i)}_{=0} +(\eps^\beta t_0+\eps^{\tilde\beta}t_i)\nabla u_0(\hat\xi_i)^T \cdot\nu(\hat\xi_i)+ \O((\eps^\beta t_0+\eps^{\tilde\beta}t_i)^2) \\&= (\eps^\beta t_0+\eps^{\tilde\beta}t_i) \partial_\nu u_0(\xi_0)+ \frac 12(\eps^\beta t_0+\eps^{\tilde\beta}t_i)\eps^{2\hat\beta}(D_{N-1}^2\partial_\nu u_0(\xi_0)\tau_i)^T\cdot \tau_i \\&+\O((\eps^\beta t_0+\eps^{\tilde\beta}t_i)^2) + \O(\eps^{2\hat\beta}(\eps^\beta t_0+\eps^{\tilde\beta}t_i)).
	\end{align*}
	Then
	$$\begin{aligned}
		-\mathtt C\sum_i \d_i^{\frac{N-2}{2}}u_0(\xi_i)&=-\mathtt C\sum_i \e^{\alpha\frac{N-2}{2}}\left(d_0^{\frac{N-2}{2}}+\frac{N-2}{2}\e^{\hat\alpha-\alpha}d_0^{\frac{N-4}{2}}d_i+\frac{(N-2)(N-4)}{4}\e^{2(\hat\alpha-\alpha)}d_0^{\frac{N-6}{2}}d_i^2\right)u_0(\xi_i)\\
		&=-\mathtt C \e^{\alpha\frac{N-2}{2}+\beta}d_0^{\frac{N-2}{2}}t_0 k \partial_\nu u_0(\xi_0)+\O(\eps^{\alpha\frac{N-2}{2}+2\beta}t_0 )
		\\&-\mathtt C \e^{\alpha\frac{N-2}{2}+\tilde\beta}d_0^{\frac{N-2}{2}}\partial_\nu u_0(\xi_0)\sum_i t_i+\sum\O(\eps^{\alpha\frac{N-2}{2}+\beta+\tilde\beta}t_i)
		\\&-\mathtt C\frac{N-2}{2} \e^{\alpha\frac{N-2}{2}+\hat\alpha-\alpha+\beta}d_0^{\frac{N-4}{2}}t_0\partial_\nu u_0(\xi_0)\sum_i d_i+\sum\O(\eps^{\alpha\frac{N-2}{2}+2\beta+\hat\alpha-\alpha}d_i)\\&
		-\frac{\mathtt C}{ 2} \e^{\alpha\frac{N-2}{2}+\beta+2\hat\beta}d_0^{\frac{N-2}{2}}t_0\sum_i (D_{N-1}^2\partial_\nu u_0(\xi_0)\tau_i)^T\cdot \tau_i+\sum\O_{\tau_i}(\eps^{\alpha\frac{N-2}{2}+\tilde\beta+2\hat\beta})\\
		&-\mathtt C\frac{N-2}{2} \e^{\alpha\frac{N-2}{2}+\hat\alpha-\alpha+\tilde\beta}d_0^{\frac{N-4}{2}}\partial_\nu u_0(\xi_0)\sum_i d_it_i+\sum\O(\eps^{\alpha\frac{N-2}{2}+\hat\alpha-\al+\beta+\tilde\beta}t_id_i)\\
		&-\mathtt C \frac{(N-2)(N-4)}{4}\e^{\alpha\frac{N-2}{2}+2\hat\alpha-2\alpha+\beta}d_0^{\frac{N-6}{2}}t_0\partial_\nu u_0(\xi_0)\sum_i d^2_i\\&+\sum\O(\eps^{\alpha\frac{N-2}{2}+2\beta+2\hat\al-2\al}d_i^2)+\sum\O(\eps^{\alpha\frac{N-2}{2}+2\tilde\beta}t_i^2)\\
	\end{aligned}$$
	where the lower order terms are $\mathcal O\left(\e^{\hat\theta+\sigma}\right)$. Now the zero order terms are
	$$\begin{aligned}&\mathtt A +\e^{(\alpha-\beta)(N-2)}\frac{\alpha_N}{2^{N-1}} \mathtt C k \frac{d_0^{N-2}}{t_0^{N-2}}-\e^{1+2\alpha}\mathtt B k d_0^2-\mathtt C \e^{\alpha\frac{N-2}{2}+\beta}d_0^{\frac{N-2}{2}}t_0 k \partial_\nu u_0(\xi_0)
		+ \eps^{\theta+\sigma} k\mathfrak f(d_0, t_0)
		\\&=\mathtt A+\e^{\theta}k \left(	\frac{\alpha_N}{2^{N-1}} \mathtt C\frac{d_0^{N-2}}{t_0^{N-2}}-\mathtt B d_0^2-\mathtt C d_0^{\frac{N-2}{2}}t_0  \partial_\nu u_0(\xi_0)\right) + \eps^{\theta+\sigma} k \mathfrak f(d_0, t_0)
		.
	\end{aligned}$$
	The first order terms are
	$$\begin{aligned}&-2\e^{1+\alpha+\hat\alpha}\mathtt B d_0\sum_i d_i +\e^{\alpha(N-3)-\beta(N-2)+\hat\alpha}\frac{\alpha_N}{2^{N-1}} \frac{d_0^{N-3}}{t_0^{N-2}}\mathtt C\sum_id_i-\mathtt C\frac{N-2}{2} \e^{\alpha\frac{N-2}{2}+\hat\alpha+\beta}d_0^{\frac{N-4}{2}}t_0\partial_\nu u_0(\xi_0)\sum_i d_i\\
		&-\e^{\alpha(N-2)-\beta(N-1)+\tilde\beta}\frac{\alpha_N(N-2)}{2^{N-1}} \frac{d_0^{N-2}}{t_0^{N-1}}\mathtt C\sum_i t_i-\mathtt C \e^{\alpha\frac{N-2}{2}+\tilde\beta}d_0^{\frac{N-2}{2}}\partial_\nu u_0(\xi_0)\sum_i t_i\\
		&=\e^{\hat\theta}\underbrace{\left(-2\mathtt B d_0+\frac{\alpha_N}{2^{N-1}}\frac{d_0^{N-3}}{t_0^{N-2}}\mathtt C-\mathtt C\frac{N-2}{2}d_0^{\frac{N-4}{2}}t_0\partial_\nu u_0(\xi_0)\right)}_{\hbox{\tiny{}} =0}\sum_i d_i\\
		&+\e^{\hat\theta}\underbrace{\left(-\frac{\alpha_N(N-2)}{2^{N-1}}\frac{d_0^{N-2}}{t_0^{N-1}}\mathtt C-\mathtt C d_0^{\frac{N-2}{2}}\partial_\nu u_0(\xi_0)\right)}_{\hbox{\tiny{}} =0}\sum_i t_i
	\end{aligned}$$
	The quantities in the previous box are zero as $d_0$ and $t_0$ satisfy the system in \eqref{sys}. 
	Then the order zero terms are a function of $(d_0, t_0)$, namely
	$$\begin{aligned}&\mathtt A+\e^{\theta}k \left(	\frac{\alpha_N}{2^{N-1}} \mathtt C\frac{d_0^{N-2}}{t_0^{N-2}}-\mathtt B d_0^2-\mathtt C d_0^{\frac{N-2}{2}}t_0  \partial_\nu u_0(\xi_0)\right)+k \e^{\theta+\sigma}\mathfrak f(d_0, t_0)
		\\&=\mathtt A+\e^\theta k\mathfrak g(d_0, t_0)+ k \e^{\theta+\sigma}\mathfrak f(d_0, t_0).
	\end{aligned}$$
	Now the second order terms are
	$$\begin{aligned}&-\e^{1+2\hat\alpha}\mathtt B\sum_i d_i^2+\e^{\alpha(N-2)-\beta N+2\tilde\beta}\frac{\alpha_N(N-2)(N-1)}{2^{N}} \frac{d_0^{N-2}}{t_0^{N}}\mathtt C\sum_i t_i^2\\
		&-\e^{\alpha(N-3)+\hat\alpha-\beta(N-1)+\tilde\beta}\frac{\alpha_N(N-2)^2}{2^{N-1}} \frac{d_0^{N-3}}{t_0^{N-1}}\mathtt C\sum_i t_id_i\\
		&+\e^{\alpha(N-4)+2\hat\alpha-\beta(N-2)}\frac{\alpha_N(N-2)(N-3)}{2^{N}} \frac{d_0^{N-4}}{t_0^{N-2}}\mathtt C\sum_i d_i^2-\e^{\alpha(N-2)-\hat\beta(N-2)}\alpha_N \mathtt C\sum_{h<i}\frac{d_0^{N-2}}{|\tau_i-\tau_h|^{N-2}}\\
		&-\frac{\mathtt C}{ 2} \e^{\alpha\frac{N-2}{2}+\beta+2\hat\beta}d_0^{\frac{N-2}{2}}t_0\sum_i (D_{N-1}^2\partial_\nu u_0(\xi_0)\tau_i)^T\cdot \tau_i\\
		&-\mathtt C\frac{N-2}{2} \e^{\alpha\frac{N-2}{2}+\hat\alpha-\alpha+\tilde\beta}d_0^{\frac{N-4}{2}}\partial_\nu u_0(\xi_0)\sum_i d_it_i\\
		&-\mathtt C \frac{(N-2)(N-4)}{4}\e^{\alpha\frac{N-2}{2}+2\hat\alpha-2\alpha+\beta}d_0^{\frac{N-6}{2}}t_0\partial_\nu u_0(\xi_0)\sum_i d^2_i\\
		&=\e^{\hat\theta}\underbrace{\left(-\mathtt B +\frac{\alpha_N(N-2)(N-3)}{2^N}\frac{d_0^{N-4}}{t_0^{N-2}}\mathtt C-\mathtt C \frac{(N-2)(N-4)}{4}d_0^{\frac{N-6}{2}}t_0\partial_\nu u_0(\xi_0)\right)}_{:=\mathfrak A}\sum_i d_i^2\\
		&+\e^{\hat\theta}\underbrace{\left(-\frac{\alpha_N(N-2)^2}{2^{N-1}}\frac{d_0^{N-3}}{t_0^{N-1}}\mathtt C-\mathtt C\frac{N-2}{2}d_0^{\frac{N-4}{2}}\partial_\nu u_0(\xi_0)\right)}_{:=\mathfrak B}\sum_i t_id_i\\
		&+\e^{\hat\theta}\underbrace{\left(\frac{\alpha_N (N-2)(N-1)}{2^N}\frac{d_0^{N-2}}{t_0^N}\mathtt C\right)}_{:=\mathfrak C>0}\sum_i t_i^2\\
		&+\e^{\hat\theta}\left(-\alpha_N\mathtt C \sum_{h<i}\frac{d_0^{N-2}}{|\tau_i-\tau_h|^{N-2}}-\frac{\mathtt C}{2}d_0^{\frac{N-2}{2}}t_0\sum_i (D_{N-1}^2\partial_\nu u_0(\xi_0)\tau_i)^T\cdot \tau_i \right)
	\end{aligned}$$
	By using the first equation of \eqref{sys} we get
	$$\mathfrak A:= \frac{\alpha_N(N^2-5N+5)}{2^N}\frac{d_0^{N-4}}{t_0^{N-2}}\mathtt C-\mathtt C\frac{(N-2)(N-5)}{4}d_0^{\frac{N-6}{2}}t_0\partial_\nu u_0(\xi_0)>0$$
	while by using the second equation of \eqref{sys} we get
	$$\mathfrak B:=0 .$$
	Then at the end
	$$\begin{aligned}J&_\e(u_0-\sum_i P\Ui)=\mathtt A+\e^{\theta} \mathfrak g(d_0, t_0)+k \e^{\theta+\sigma}\mathfrak f(d_0, t_0)\\
		&+\e^{\hat\theta}\left(\mathfrak A \sum_i d_i^2+\mathfrak C\sum_i t_i^2 -\alpha_N\mathtt C \sum_{h<i}\frac{d_0^{N-2}}{|\tau_i-\tau_h|^{N-2}}-\frac{\mathtt C}{2}d_0^{\frac{N-2}{2}}t_0\sum_i (D_{N-1}^2\partial_\nu u_0(\xi_0)\tau_i)^T\cdot \tau_i\right)\\
		&+\mathcal O\left(\e^{\hat\theta+\sigma}\right).\end{aligned}$$
\end{proof}

Now standard arguments permit us to conclude that if $(\bm d_\eps,\bm t_\eps, \bm \tau_\eps)$ is a critical point of $\tilde J_\eps$, then 
	\[W_{\bm\de_\eps,\bm\xi_\eps}+\phi_{\bm\de_\eps,\bm\xi_\eps}= u_0-\sum_{h=1}^k P\U_{{{\de}_h}_\eps,{{\xi}_h}_\eps}+\phi_{\bm\de_\eps,\bm\xi_\eps}\] is a solution of \eqref{BN} where $\bm \de_\eps$ and $ \bm \xi_\eps$ are as in Section \ref{ans}.

Indeed we can rewrite the problem \eqref{eq pi} as
\begin{equation} \label{cijs}
	u-\i(f(u)+\eps u) = \sum c_i^j P\psi_i^j    
\end{equation} and our goal is to find appropriate parameters $\bm d=(d_1,\cdots,d_k),\bm t=(t_1,\cdots, t_k)\in \RR^k$ and points $\bm \tau=(\tau_1,\cdots,\tau_k)\in(\RR^{N-1})^k$ such that  $c_i^j=c_i^j(\bm d,\bm t, \bm \tau)=0$ for all $i=1,\cdots,k$ and $j=0,\cdots,N$.

\begin{prop}\label{cij null}
	If  for all $i=1,\cdots,k$ and $j=1,\cdots N$
	\begin{equation} \label{hp} DJ_\eps(u)[\partial_{\de_i}u] = DJ_\eps(u)[\partial_{(\xi_i)_j}u] = 0,\end{equation} then $c_i^j=0$.
\end{prop}
\begin{proof}
	It's easy to see that
	\[
		D J_\eps(u)[f]=\langle u-\i (f(u)+\eps)), f \rangle .
	\] Then by \eqref{hp} follows
	\[\begin{cases}
			&\langle u-\i (f(u)+\eps)),  \partial_{\delta_i}u  \rangle = 0 
			\\&  \langle u-\i (f(u)+\eps)), \nabla_{\xi_i} u \rangle = 0      
		\end{cases}
	\]
	and combining it with \eqref{cijs}, for all $i=1,\cdots, k$, $r=1,\cdots,N$ we have \[
		\sum_{h,j} c_h^j \int_\Omega \U_h^{p-1}\psi_h^j \partial_{\delta_i}u =      \sum_{h,j} c_h^j \int_\Omega \U_h^{p-1}\psi_h^j \partial_{(\xi_i)_r}u = 0. 
	\]
	We have the following estimates on the integrals
	\[\begin{aligned}
			&\delta_i \int_\Omega \U_h^{p-1}\psi_h^j \partial_{\delta_i}PU_i = \begin{cases}
				b_1 +{o(1)} &\mbox{ if } i=h, j=0
				\\ {o(1)} &\mbox{ otherwise }
			\end{cases} ; \\
			&\delta_i\int_\Omega \U_h^{p-1}\psi_h^j \partial_{(\xi_i)_r}PU_i = \begin{cases}
				b_2 + o(1) &\mbox{ if } i=h, r=j  \\ o(1) &\mbox{ otherwise } 
			\end{cases}
		\end{aligned}
	\]
	where $b_1,b_2\neq 0$. {Recalling that $\phi$ satisfies \eqref{problem}, by the Implicit Function Theorem we can prove that there exist $\partial_{\delta_i} \phi, \nabla_{\xi_i}\phi \in \H$. In particular using \eqref{L in} and proceeding as in Proposition \ref{prop L inv}, it is possible to show that}
 \begin{equation}
     \delta_i \|\partial_{\delta_i}\phi\|_\H = o(1); \qquad 
     \delta_i \|\partial_{(\xi_i)_j}\phi\|_\H = o(1)
 \end{equation} for all $i=1,\cdots, k$ and $j=1,\cdots,N$. Then
	\begin{align*}
		& \delta_i  \int_\Omega \U_h^{p-1}\psi_h^j \partial_{\delta_i}\phi 
  =  o(1); \\ 
		& \delta_i \int_\Omega \U_h^{p-1}\psi_h^j \partial_{(\xi_i)_r}\phi 
  =o(1).
	\end{align*}
	The linear system in the $c_i^j$'s has the only possible solution  $c_i^j=0$ for all $i=1,\cdots,k$ and $j=0,\cdots,N$.
\end{proof}

Now we are able to conclude our proof.

\begin{proof}[Proof of Theorem \ref{main thm}] 
	By Proposition \ref{prop energy}, we have that 
	\[\begin{aligned}
		\tilde J_\eps (\bm d,\bm t,\bm \tau) 
		&=\e^{\hat\theta} \Phi(\bm\de,\bm t,\bm \tau)+\mathtt A_1
		+\mathcal O\left(\e^{\hat\theta+\sigma}\right) \end{aligned}\]
	where $A_1$ depends only on $(d_0,t_0,\xi_0)$ and 
	\[ \Phi (\bm d,\bm t,\bm \tau)=\mathfrak A \sum_i d_i^2+\mathfrak C\sum_i t_i^2 -\alpha_N\mathtt C \sum_{h<i}\frac{d_0^{N-2}}{|\tau_i-\tau_h|^{N-2}}-\frac{\mathtt C}{2}d_0^{\frac{N-2}{2}}t_0\sum_i  (D_{N-1}^2\partial_\nu u_0(\xi_0)\tau_i)^T\cdot \tau_i .\]
	Since the constants  $\mathfrak A, \mathfrak C$ are strictly positive,  $\bm d=\bm t = (0,\cdots,0) \in \RR^{2k} $ is  the unique critical point for $\Phi$ in $\bm d$ and $\bm t$.
	Since the matrix $D_{N-1}^2 \partial_\nu u_0 (\xi_0) $ is positive definite there also exists a critical point $\bm \tau_0$ in $\bm \tau$. The point $(\bm 0,\bm 0,\bm \tau)$ is a critical point for $\Phi$, which is stable under small perturbation of the function. Hence for all $\eps$ sufficiently small  there exists $(\bm\de_\eps,\bm t_\eps,\bm \tau_\eps)$, satisfying \eqref{par}, critical point for $\tilde J_\eps$, which is close to $(\bm 0,\bm 0,\bm \tau)$.

Now  observe that for all $i=1,\cdots,k$ and $ r=1,\cdots,N-1$
	\begin{align*} \begin{cases}
			&\frac{\partial}{\partial_{\delta_i}} P\U_i = \frac{\partial}{\partial_{d_i}} P\U_i \frac{\partial}{\partial_{\delta_i}} d_i; \\
			& \frac{\partial}{\partial_{(\xi_i)_r}} P\U_i = \frac{\partial}{\partial_{t_i}} P\U_i \frac{\partial}{\partial_{(\xi_i)_r}} t_i + \sum_{r=1}^{N-1}\frac{\partial}{\partial_{(\tau_i)_r}} P\U_i \frac{\partial}{\partial_{(\xi_i)_r}} (\tau_i)_r .
		\end{cases}
	\end{align*} 
	Hence the assumptions of Proposition \ref{cij null} are satisfied and we conclude that $u_\eps= W_{\bm\de_\eps,\bm\xi_\eps}+\phi_{\bm\de_\eps,\bm\xi_\eps}$ is a solution of \eqref{BN}.
\end{proof}

\section*{Data Availability Statements}
All data generated or analysed during this study are included in this article.



\normalcolor

\bibliography{biblio}
\bibliographystyle{abbrv}
 
\end{document}